\documentclass{amsart}
\usepackage{amsmath,amssymb,amsthm,amsfonts,enumerate,array,color,lscape,fancyhdr,layout,pst-all}
\usepackage[english]{babel}
\usepackage[latin1]{inputenc}
\usepackage{young}

\usepackage[hcentermath]{youngtab}
\usepackage{graphicx}
\usepackage{float}
\usepackage{pstricks}
\usepackage[all]{xy}
\usepackage{ulem}

\newcounter{numberofremark}

\newcommand\nothing[1]{}
\usepackage[active]{srcltx}
\usepackage[hyperindex,bookmarksnumbered,plainpages]{hyperref}

\newcommand{\dcl}{\DeclareMathOperator}
\dcl\cdet{cdet} \dcl\Sp{Specm} \dcl\depth{depth} \dcl\im{Im} \dcl\Span{span} \dcl\Ker{Ker} \dcl\Specm{Specm}
\dcl\Supp{Supp} \dcl\codim{codim} \dcl\Y{Y} \dcl\gl{\mathfrak{gl}}    \dcl\U{U} \dcl\T{T}
\dcl\qdet{qdet} \dcl\sgn{sgn} \dcl\gr{gr} \dcl\diag{diag}
\dcl\g{\mathfrak{g}} \dcl\C{\mathbb C} \dcl\dd{{\mathrm d}}

\newcommand\sm{{\mathsf m}}

\newcommand\Ga{{\Gamma}}

\newlength\yStones
\newlength\xStones
\newlength\xxStones

\makeatletter
\def\Stones{\pst@object{Stones}}
\def\Stones@i#1{%
  \pst@killglue%
  \begingroup%
  \use@par%
  \setlength\xxStones{\xStones}%
  \expandafter\Stones@ii#1,,\@nil
  \endgroup
  \global\addtolength\xStones{0.6cm}%
  \global\addtolength\yStones{-7.5mm}}%
\def\Stones@ii#1,#2,#3\@nil{%
  \rput(\xxStones,\yStones){%
    \psframebox[framesep=0]{%
      \parbox[c][6mm][c]{11mm}{\makebox[11mm]{$#1$}}}}%
  \addtolength\xxStones{1.2cm}%
  \ifx\relax#2\relax\else\Stones@ii#2,#3\@nil\fi}
\makeatother

\def\Stone#1{\fbox{\makebox[10mm]{\strut#1}}\kern2pt}

\newtheorem{theorem}{Theorem}[section]
\newtheorem{lemma}[theorem]{Lemma}
\newtheorem{corollary}[theorem]{Corollary}
\newtheorem{condition}[theorem]{Condition}
\newtheorem{proposition}[theorem]{Proposition}
\newtheorem{example}[theorem]{Example}
\newtheorem{remark}[theorem]{Remark}

\newtheorem{conjecture}[theorem]{Conjecture}
\newtheorem{definition}[theorem]{Definition}

\begin{document}
\title{Combinatorial construction of Gelfand-Tsetlin modules for $\mathfrak{gl}_n$}
\author{Vyacheslav Futorny}
\address{Instituto de Matem\'atica e Estat\'istica, Universidade de S\~ao
Paulo,  S\~ao Paulo SP, Brasil}\email{futorny@ime.usp.br}
\author{Luis Enrique Ramirez}
\address{Universidade Federal do ABC, Santo Andr\'e SP, Brasil} \email{luis.enrique@ufabc.edu.br}
\author{Jian Zhang}
\address{Instituto de Matem\'atica e Estat\'istica, Universidade de S\~ao
Paulo,  S\~ao Paulo SP, Brasil} \email{zhang@ime.usp.br}

\begin{abstract}
We propose a new effective method of constructing explicitly  \\Gelfand -Tsetlin modules for $\gl_n$.  We obtain a large family of simple   modules  that have a basis consisting of
Gelfand-Tsetlin tableaux, the action of the Lie algebra is given by the  Gelfand-Tsetlin formulas and with all Gelfand-Tsetlin multiplicities equal $1$.  As an application of our construction we prove
 necessary and sufficient condition  for the Gelfand and Graev's continuation construction to define a module which was conjectured by Lemire and Patera.

\end{abstract}


\maketitle
\section{Introduction}

A classical paper of  Gelfand  and  Tsetlin \cite{GT}
describes a basis of simple finite dimensional modules over the Lie algebra $\gl_n$.  This is one of the most remarkable results of the representation theory of  Lie algebras which triggered
a strong interest and initiated a development of  the theory of Gelfand-Tsetlin modules in  \cite{DFO}, \cite{Gr1}, \cite{Gr2}, \cite{Maz1}, \cite{Maz2}, \cite{m:gtsb}, \cite{Zh}, among others.
Gelfand-Tsetlin representations are related to  Gelfand-Tsetlin integrable systems studied by Guillemin and Sternberg \cite{GS}, Kostant and Wallach \cite{KW1}, \cite{KW2}, Colarusso and Evens
\cite{CE1}, \cite{CE2}.  Each tableau in the basis of a finite dimensional representation is an eigenvector  of  the Gelfand-Tsetlin subalgebra $\Gamma$,  certain maximal commutative subalgebra of the universal enveloping algebra of $\gl_n$. Hence, any such  tableau corresponds to a maximal ideal of $\Gamma$.
Gelfand-Tsetlin theory had a successful development for infinite dimensional representations in \cite{O}, \cite{FO} where it was shown that simple Gelfand-Tsetlin modules are parametrized up to some finiteness by the maximal ideals of $\Gamma$.
The significance of the class of Gelfand-Tsetlin modules is in the fact that they form the largest subcategory of $\gl_n$-modules  (in particular weight modules with respect to a fixed Cartan subalgebra) where there is some understanding of simple modules. The main remaining problem
is how to construct explicitly these modules.

There were essentially two main approaches to generalize Gelfand-Tsetlin basis and construct explicitly new simple modules. One, starting from \cite{DFO},
was aiming to construct {\it generic} Gelfand-Tsetlin modules, that is those having a basis consisting of tableaux with no integer differences between the entries of the same row with the exception of the top row. These simple  modules were described in \cite{FGR2}.
Next step was to consider
$1$-singular case when there is just one pair in only one row with integer difference. A break through was  a  paper \cite{FGR3} (see also \cite{FGR4}, \cite{Za}, \cite{Vi1})
 where such modules were explicitly constructed, followed by  general constructions of certain "universal" modules in \cite{RZ} and \cite{Vi2}.  

A different approach in constructing new simple modules is due to Gelfand and Graev \cite{GG}  who
 presented a systematic study of formal analytic continuations
of both the labelling and the algebra structure of finite dimensional representations of $\mathfrak{gl}_n$.
Imposing certain conditions  on entries of a  tableau Gelfand and Graev described new infinite dimensional
simple modules with a basis consisting of tableaux and the algebra action given by the classical Gelfand-Tsetlin formulas.  We will call this condition
 {\it GG-condition}.
However,
Lemire and Patera \cite{LP} showed that
some of these analytic continuations, in fact,  are not representations. They conjectured
 a necessary and sufficient condition, here called the {\it LP-condition}, for Gelfand-Graev continuation to define a module and proved it for $\gl_3$ and $\gl_4$ (partial cases).

 Our first result establishes this conjecture.

 \

 \noindent {\bf Theorem I.} Gelfand-Graev continuation defines a $\gl_n$-module if and only if it satisfies the LP condition.

 \

We propose a new combinatorial method of constructing simple Gelfand-Tsetlin modules by continuation
from tableau satisfying more general conditions which we call, following the tradition,
{\it FRZ-condition}.  Any tableau satisfying the LP-condition also satisfies the FRZ-condition but the latter is much more general.
Each tableau $L$ satisfying the FRZ-condition defines  the maximal ideal $\sm_L$ of the Gelfand-Tsetlin subalgebra
and
 a  Gelfand-Tsetlin module $V(L)$ with ideal $\sm_L$  in its Gelfand-Tsetlin support.
 A tableau $L$ is {\it critical} if it has equal entries in one or more rows different from the top row. Otherwise, tableau is noncritical.
 Our second main result is

\

 \noindent {\bf Theorem II.}  Let $L$ be a tableau  satisfying the FRZ-condition. There exists a unique simple Gelfand-Tsetlin
 $\gl_n$-module  $V(L)$ with maximal ideal $\sm_L$ in its Gelfand-Tsetlin support having the following properties:
 \begin{itemize}
 \item $V(L)$ has a basis consisting of noncritical tableaux with standard action of the generators of $\gl_n$.
 \item All Gelfand-Tsetlin multiplicities of $V(L)$ are bounded by $1$.
 \end{itemize}

\

We will call the class of  Gelfand-Tsetlin $\gl_n$-modules from Theorem II {\it admissible} modules.
Hence, we have a combinatorial way to explicitly construct a vast number of new simple  Gelfand-Tsetlin modules.
Some examples of admissible modules were constructed by Mazorchuk \cite{Maz1}, \cite{Maz3}.\\

 It is interesting to know the place of admissible modules in the category of all
Gelfand-Tsetlin modules. We state the following conjecture\\

 \noindent {\bf Conjecture 1.} If $V$ is simple  Gelfand-Tsetlin $\gl_n$-module with Gelfand-Tsetlin multiplicities $1$ which has  a basis consisting of noncritical
 tableaux with standard action of the generators of $\gl_n$ then $V\simeq V(L)$ for some tableau $L$ satisfying the FRZ-condition.\\
At the moment  Conjecture 1 is known to be true for $n\leq 3$, see also Conjecture \ref{conj-tableau}. \\


\

\noindent{\bf Acknowledgements.}  V.F. is
supported in part by  CNPq  (304467/2017-0) and by
Fapesp  (2014/09310-5). L.E.R. is supported by 
Fapesp grant (2018/17955-7).
J. Z. is supported by  Fapesp grant (2015/05927-0).

\section{Notation and conventions}

Throughout the paper we fix an integer $n\geq 2$. The ground field will be ${\mathbb C}$.  For $a \in {\mathbb Z}$, we write $\mathbb Z_{\geq a}$ for the set of all integers $m$ such that $m \geq a$. Similarly, we define $\mathbb Z_{< a}$.   By $\gl_n$ we denote the general linear Lie
algebra consisting of all $n\times n$ complex matrices, and by $\{E_{i,j}\mid 1\leq i,j \leq n\}$  - the
standard basis of $\gl_n$ of elementary matrices.
As a Lie algebra, $\gl_n$ is  generated by $E_{i,i+1}, E_{i+1,i},1\leq i\leq n-1$ and $E_{i,i}, 1\leq i\leq n$.
We fix the standard Cartan subalgebra  $\mathfrak h$, the standard triangular decomposition and the corresponding basis of simple roots of $\gl_n$.  The weights of $\gl_n$ will be written as $n$-tuples $(\lambda_1,...,\lambda_n)$.

Let  $\mathfrak{Gl}_n$ be the free Lie algebra on generators $e_{i},f_{i},1\leq i\leq n-1$, and $H_{i},1\leq i\leq n$. There is a homomorphism $\mathfrak{Gl}_n\rightarrow \gl_n$ given by
$e_i\mapsto E_{i,i+1}, f_i\mapsto E_{i+1,i},1\leq i\leq n-1 $ and $H_i\mapsto E_{ii},1\leq i\leq n$. The kernel $\mathfrak{k}_n$ of this homomorphism is generated by the following elements
\begin{align}
 &[H_i,H_j], \quad 1\leq i,j\leq n,\label{h_i relation}\\
&[e_i,f_j]-\delta_{ij}(H_i-H_{i+1}),\quad 1\leq i,j\leq n-1,\label{e_i, f_j relation}\\
&[H_j,e_i]-(\delta_{ij}-\delta_{i+1,j})e_i, \label{h with e_i relation}\\
&[H_i,f_j]+(\delta_{ij}-\delta_{i+1,j})f_j,\label{h with f_j relation}\\
&[e_i,[e_i,e_j]], [f_i,[f_i,f_j]]  \quad ,\quad 1\leq i,j\leq n-1,|i-j|=1,\label{Serre relations}\\
&[e_i,e_j],[f_i,f_j]  \quad ,\quad 1\leq i,j\leq n-1,|i-j|>1\label{ee and ff relations}.
\end{align}

Given a Lie algebra $\mathfrak a$ we denote its universal enveloping algebra by $U(\mathfrak a)$. Throughout the paper we abbreviate $U = U(\gl_n)$.
For a commutative ring $R$, by ${\rm Specm}\, R$ we denote the set of maximal ideals of $R$. We will write vectors in $\mathbb{C}^{\frac{n(n+1)}{2}}$ as $
L=(l_{ij})=(l_{n1},...,l_{nn}|l_{n-1,1},...,l_{n-1,n-1}| \cdots|l_{21}, l_{22}|l_{11})$. We denote ${\mathbb Z}_0^{\frac{n(n+1)}{2}}:=\{z\in {\mathbb Z}^{\frac{n(n+1)}{2}}|z_{ni}=0,1\leq i\leq n\}$. For $1\leq j \leq i \leq n$, $\boldsymbol{\delta}^{ij} \in{\mathbb Z}_0^{\frac{n(n+1)}{2}}$ is defined by  $(\boldsymbol{\delta}^{ij})_{ij}=1$ and all other $(\boldsymbol{\delta}^{ij})_{k\ell}$ are zero.

For $i>0$ by $S_i$ we denote the $i$th symmetric group.  Throughout the paper we set  $G:=S_n\times\cdots \times S_1$.
 Finally, given $a,b\in\mathbb{C}$, we will write $a\geq b\ (\text{respectively }a> b)$ if $a-b\in\mathbb{Z}_{\geq0}$ (respectively $\mathbb{Z}_{>0}$).

\section{Gelfand-Tsetlin Theorem and Gelfand-Tsetlin modules}

In 1950, I Gelfand and M. Tsetlin gave an explicit realization of all simple finite dimensional modules for $\mathfrak{gl}_n$. Let us recall the construction.

\begin{definition}
Given $L=(l_{ij})\in\mathbb{C}^{\frac{n(n+1)}{2}}$, denote by $T(L)$ the array

\begin{center}

\Stone{\mbox{ \scriptsize {$l_{n1}$}}}\Stone{\mbox{ \scriptsize {$l_{n2}$}}}\hspace{1cm} $\cdots$ \hspace{1cm} \Stone{\mbox{ \scriptsize {$l_{n,n-1}$}}}\Stone{\mbox{ \scriptsize {$l_{nn}$}}}\\[0.2pt]
\Stone{\mbox{ \scriptsize {$l_{n-1,1}$}}}\hspace{1.5cm} $\cdots$ \hspace{1.5cm} \Stone{\mbox{ \tiny {$l_{n-1,n-1}$}}}\\[0.3cm]
\hspace{0.2cm}$\cdots$ \hspace{0.8cm} $\cdots$ \hspace{0.8cm} $\cdots$\\[0.3cm]
\Stone{\mbox{ \scriptsize {$l_{21}$}}}\Stone{\mbox{ \scriptsize {$l_{22}$}}}\\[0.2pt]
\Stone{\mbox{ \scriptsize {$l_{11}$}}}\\
\medskip
\end{center}
Such an array will be called a \it{Gelfand-Tsetlin tableau} of height $n$. A Gelfand-Tsetlin tableau of height $n$ is called \it{standard} if
$l_{ki}-l_{k-1,i}\in\mathbb{Z}_{\geq 0}$ and $l_{k-1,i}-l_{k,i+1}\in\mathbb{Z}_{> 0}$ for all $1\leq i\leq k\leq  n$.
\end{definition}

\begin{theorem}[\cite{GT}]\label{Gelfand-Tsetlin theorem}
Let $\lambda=(\lambda_{1},\ldots,\lambda_{n})$ be an integral dominant $\mathfrak{gl}_n$-weight. The vector space $L(\lambda)$ spanned by all standard tableaux $T(L)$ with fixed top row $l_{nj}=\lambda_j-j+1$ is a $\mathfrak{gl}_{n}$-module. The action of $\mathfrak{gl}_{n}$ on $L(\lambda)$ is given by the following explicit formulae. 

\begin{equation}\label{Gelfand-Tsetlin formulas}
\begin{split}
E_{k,k+1}(T(L))&=-\sum_{i=1}^{k}\left(\frac{\prod_{j=1}^{k+1}(l_{ki}-l_{k+1,j})}{\prod_{j\neq i}^{k}(l_{ki}-l_{kj})}\right)T(L+\boldsymbol{\delta}^{ki}),\\ E_{k+1,k}(T(L))&=\sum_{i=1}^{k}\left(\frac{\prod_{j=1}^{k-1}(l_{ki}-l_{k-1,j})}{\prod_{j\neq i}^{k}(l_{ki}-l_{kj})}\right)T(L-\boldsymbol{\delta}^{ki}),\\
E_{kk}(T(L))&=\left(k-1+\sum_{i=1}^{k}l_{ki}-\sum_{i=1}^{k-1}l_{k-1,i}\right)T(L),
\end{split}
\end{equation}
where we use the convention that if the new  tableau $T(L\pm\boldsymbol{\delta}^{ki})$ is not standard, then the corresponding summand of $E_{k,k+1}(T(L))$ or $E_{k+1,k}(T(L))$ is zero by definition. Moreover, every simple finite dimensional $\mathfrak{gl}_{n}$-module  is isomorphic to $L(\lambda)$ for some $ \lambda$.
\end{theorem}

 For $m\leqslant n$, $\mathfrak{gl}_{m}$ naturally identifies with the
subalgebra of $\gl_n$ spanned by $\{ E_{ij}\,|\, 1\leq i,j\leq m \}$.
We have the following chain
$$\gl_1\subset \gl_2\subset \ldots \subset \gl_n.$$
It induces  the chain $U(\gl_{1})\subset$ $U(\gl_{2})$ $\ldots$ $\subset
U(\gl_{n})$ for the universal enveloping algebras. Let
$Z_{m}$ be the center of $U(\gl_{m})$, $1\leq m\leq n$.
The subalgebra ${\Ga}$ of $U$ generated by the $
Z_m, 1\leq m\leq n  $ is called its  Gelfand-Tsetlin subalgebra (cf \cite{DFO}).

We define now our main objects.

\begin{definition}
\label{definition-of-GZ-modules} A finitely generated $U$-module
$M$ is called a \it{Gelfand-Tsetlin module } if

\begin{equation}\label{equation-Gelfand-Tsetlin-module-def}
M=\bigoplus_{\sm\in\Sp\Ga}M(\sm),
\end{equation}

where $$M(\sm)=\{v\in M\ |\ \sm^{k}v=0 \text{ for some }k\geq 0\}.$$
The {\it Gelfand-Tsetlin support} of $M$ is the set
$\Supp_{GT}(M):=\{\sm\in\Sp\Ga\ |\ M(\sm)\neq 0\}$. The {\it Gelfand-Tsetlin multiplicity} of $\sm$ in $M$ is the dimension of $M(\sm)$.
\end{definition}

The category of  Gelfand-Tsetlin modules is a full subcategory of the category of (weight) $\gl_n$-modules. It is closed under the operations of taking submodules, quotients, finite extensions and finite direct sums (cf. \cite{DFO}, \cite{FGR3}).


The elements $\{c_{mk}\}_{1\leq k\leq m\leq n}$ defined by 
 \begin{equation}\label{equ_3}
c_{mk } \ = \ \displaystyle {\sum_{(i_1,\ldots,i_k)\in \{
1,\ldots,m \}^k}} E_{i_1 i_2}E_{i_2 i_3}\ldots E_{i_k i_1}.
\end{equation}
generate $\Gamma$ (see \cite{Zh} Chap. IX, Section $59$). Let $\Lambda$ be the polynomial
algebra in the variables $\{\lambda_{ij}\,|$ $1\leqslant j\leqslant
i\leqslant n \}$.
The
  natural
action of the symmetric group  $S_i$ on $\{\lambda_{ij}\,|$ $1\leqslant j\leqslant
i\}$ induces an action of $G$ on $\Lambda$. There is a natural embedding $\imath:{\Ga}{\longrightarrow}$ $\Lambda$
given by $\imath(c_{mk})=\gamma_{mk}(\lambda)$ where
\begin{equation}\label{def-gamma}
\gamma_{mk}(\lambda) \  = \ \sum_{i=1}^m
(\lambda_{mi}+m-1)^k \prod_{j\ne i} \left( 1 -
\frac{1}{\lambda_{mi}-\lambda_{mj}} \right).
\end{equation}
Hence, $\Gamma$ can be identified with $G-$invariant polynomials  in $\Lambda$.

\begin{remark}\label{maximal ideals differ by permutations}
In what follows, we will identify the set $\Specm \Lambda$ of maximal ideals of $\Lambda$ with the set  $\mathbb{C}^{\frac{n(n+1)}{2}}$. If   $\pi: \Sp \Lambda \rightarrow
\Sp \Ga$, for every maximal ideal $m\in Specm \Gamma$, $\pi^{-1}(m)$ is a single $G$-orbit and in particular is finite.
\end{remark}

\begin{remark}
It was shown in \cite{Zh}  that the action of the generators $c_{rs}$ of $\Gamma$ on any basis tableau of a simple finite dimensional module is given by
\begin{equation}\label{action of Gamma in finite dimensional modules}
c_{rs}(T(L))=\gamma_{rs}(L)T(L),
\end{equation}
where the polynomials $\gamma_{rs}(L)$ are defined in (\ref{def-gamma}). In particular, any simple finite dimensional module is a Gelfand-Tsetlin module and
the Gelfand-Tsetlin subalgebra is diagonal in
 the tableaux
basis described in Theorem \ref{Gelfand-Tsetlin theorem}.
\end{remark}

Our goal is to construct explicitly new families of simple Gelfand-Tsetlin modules of $\gl_n$.
 Our approach  involves constructing of certain admissible sets of relations analogous to  relations that define standard tableaux. For each  such set of relations we   define an infinite
 family of non isomorphic Gelfand-Tsetlin  modules.

\section{Admissible relations}
\subsection{Sets of relations and realizations}
Set $\mathfrak{V}:=\{(i,j)\ |\ 1\leq j\leq i\leq n\}$. In this section we will consider certain binary relations on $\mathfrak{V}$ which are connected with the realization of simple finite dimensional modules described in Theorem \ref{Gelfand-Tsetlin theorem}. Set
\begin{align}
\mathcal{R}^+ &:=\{((i,j);(i-1,t))\ |\ 1\leq j\leq i,\ 2\leq i\leq n,\ 1\leq t\leq i-1\}\\
\mathcal{R}^- &:=\{((i,j);(i+1,s))\ |\ 1\leq j\leq i\leq n-1,\ 1\leq s\leq i+1\}\\
\mathcal{R}^{0}&:=\{((n,i);(n,j))\ |\ 1\leq i\neq j\leq n\}
\end{align}

and let $\mathcal{R}:=\mathcal{R}^{-}\cup\mathcal{R}^{0}\cup\mathcal{R}^{+}\subset \mathfrak{V}\times\mathfrak{V}$. From now any $\mathcal{C}\subseteq \mathcal{R}$ will be called a {\it set of relations}.

Associated with any $\mathcal{C}\subseteq \mathcal{R}$ we can construct a directed graph $G(\mathcal{C})$ with set of vertices  $\mathfrak{V}$ and an arrow going from $(i,j)$ to $(r,s)$ if and only if $((i,j);(r,s))\in\mathcal{C}$. For convenience we will picture the vertex set as disposed in a triangular arrangement with $n$ rows and $k$-th row given by $\{(k,1), \ldots, (k,k)\}$.

\begin{definition}
Let $\mathcal{C}$ be any set of relations.
\begin{itemize}
\item[(i)] We denote $\mathfrak{V}(\mathcal{C})\subseteq \mathfrak{V}$ the set of all vertices in $G(\mathcal{C})$ which are   starting or  ending vertices of an arrow.
\item[(ii)] $\mathcal{C}$ is called indecomposable if $G(\mathcal{C})$ is a connected graph.
\item[(iii)] $\mathcal{C}$ is called a loop if $G(\mathcal{C})$ is an oriented cycle.
\item[(iv)] Given $(i,j),\ (r,s)\in\mathfrak{V}$ we will write $(i,j)\succeq_{\mathcal{C}}  (r,s)$ if there exists a path in $G(\mathcal{C})$ starting in $(i,j)$ and finishing in $(r,s)$.
\end{itemize}
\end{definition}

Note that any $\mathcal{C}\subseteq \mathcal{R}$ can be written in the form $\mathcal{C}=\mathcal{C}^{-}\cup\mathcal{C}^{0}\cup\mathcal{C}^{+}$, where $\mathcal{C}^{-}:=\mathcal{R}^{-}\cap\mathcal{C}$, $\mathcal{C}^{0}:=\mathcal{R}^{0}\cap\mathcal{C}$ and $\mathcal{C}^{+}:=\mathcal{R}^{+}\cap\mathcal{C}$. Now we are ready to see how the sets of relations we are considering are generalizations of the standard sets of relations used in Theorem \ref{Gelfand-Tsetlin theorem}.

\begin{definition}
Let $\mathcal{C}$ be any set of relations and $T(L)$ any Gelfand-Tsetlin tableau.

\begin{itemize}
\item[(i)] We will say that $T(L)$ satisfies $\mathcal{C}$ if:
\begin{itemize}
\item[$\bullet$] $l_{ij}-l_{rs}\in \mathbb{Z}_{\geq 0}$ for any $((i,j); (r,s))\in \mathcal{C}^+\cup\mathcal{C}^0$.
\item[$\bullet$] $l_{ij}-l_{rs}\in \mathbb{Z}_{> 0}$ for any $((i,j); (r,s))\in \mathcal{C}^-$.
\end{itemize}
\item[(ii)] We say that $T(L)$ is a $\mathcal{C}$-realization if $T(L)$ satisfies $\mathcal{C}$ and for any $1\leq k\leq n-1$ we have, $l_{ki}-l_{kj}\in \mathbb{Z} $ if and only if $(k,i)$ and $(k,j)$ in the same connected component of $G(\mathcal{C})$.
\item[(iii)] Suppose that $T(L)$ satisfies $\mathcal{C}$. By ${\mathcal B}_{\mathcal{C}}(T(L))$ we denote the set of all tableaux of the form $T(L+z)$, $z\in {\mathbb Z}_0^{\frac{n(n+1)}{2}}$ satisfying $\mathcal{C}$. By $V_{\mathcal{C}}(T(L))$ we denote the complex vector space spanned by ${\mathcal B}_{\mathcal{C}}(T(L))$.
\end{itemize}
\end{definition}

If $(n, i)\succeq_{\mathcal{C}_0}(n, i)$ for $1\leq i\neq j\leq n$, we call the corresponding loop trivial,  otherwise the loop is called nontrivial. If $\mathcal{C}$ contains a trivial loop then there exists $\mathcal{C}'\subseteq \mathcal{C}$ such that
any $\mathcal{C}$-realization $T(L)$ is a $\mathcal{C}'$-realization  and
${\mathcal B}_{\mathcal{C}'}(T(L))={\mathcal B}_{\mathcal{C}}(T(L))$. Therefore for convenience throughout this paper we only consider sets of relations $\mathcal{C}$ which do not contain trivial loops.

\begin{example}\label{standard and generic relations}
A tableau $T(L)$ is standard if and only if $T(L)$ satisfies the set of relations $\mathcal{S}=\mathcal{S}^+\cup\mathcal{S}^-$ where
\begin{align*}
\mathcal{S}^+ &:=\{(i+1,j);(i,j))\ |\ 1\leq j\leq i\leq n-1\}\\\mathcal{S}^- &:=\{((i,j);(i+1,j+1))\ |\ 1\leq j\leq i\leq n-1\}.
\end{align*}
Note that any Gelfand-Tsetlin tableau satisfies $\emptyset$ and $T(L)$ is a $\emptyset$-realization if and only if $l_{ki}-l_{kj}\notin\mathbb{Z}$ for any $1\leq k\leq n-1$.
\end{example}


Our goal is to determine for which sets of relations $\mathcal{C}$ and tableaux $T(L)$ one can define a $\gl_n$-module structure on $V_{\mathcal{C}}(T(L))$ with the action of $\gl_n$ given by the Gelfand-Tsetlin formulas.

\begin{definition}
Let $\mathcal{C}$ be a subset of $\mathcal{R} $. We call  $\mathcal{C}$  \it{admissible} if for any
$\mathcal{C}$-realization $T(L)$, the
 Gelfand-Tsetlin formulas (\ref{Gelfand-Tsetlin formulas}) define on $V_{\mathcal{C}}(T(L))$ a structure of $\gl_n$-module.
 
 \end{definition}

It follows from Theorem \ref{Gelfand-Tsetlin theorem} that $\mathcal{S}$ defined in Example \ref{standard and generic relations} is admissible. Moreover, finite dimensional modules can be described by considering  subsets of $\mathcal{R}^{-}\cup\mathcal{R}^{+}$. The following example
 justifies why we need to consider relations involving $\mathcal{R}^{0}$.

\begin{example} Let $\mathcal{C}$ be the set of relations with associated graph

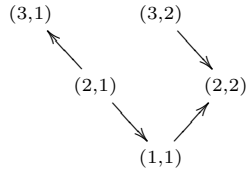
\begin{center}\begin{tabular}{c c}
\xymatrixrowsep{0.5cm}
\xymatrixcolsep{0.1cm}\xymatrix @C=0.2em {
 \scriptstyle{(3,1)}& &\scriptstyle{(3,2)}\ar[rd]   & &    \\
  & \scriptstyle{(2,1)}  \ar[lu]\ar[rd]   & &\scriptstyle{(2,2)}   \\
    &     &\scriptstyle{(1,1)}\ar[ru]  &   \\
}
\end{tabular}
\end{center}
If $T(L)$ is a $\mathcal{C}$-realization, then by Theorem \ref{sufficiency of admissible} and by Propositions \ref{not admissible}  one has:
\begin{itemize}
\item[(i)]  $V_{\mathcal{C}}(T(L))$ is a module if  $l_{31}\geq l_{32}-1$.
\item[(ii)]   $V_{\mathcal{C}}(T(L))$ is not a module if  $l_{31}< l_{32}-1$.
\end{itemize}

{
If $l_{31}\geq l_{32} $, all tableaux in $V_{\mathcal{C}}(T(L))$ satisfy the relations associated with the following graph
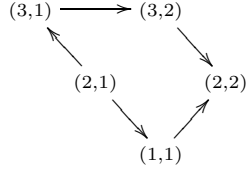
\begin{center}
\begin{tabular}{c c}
\xymatrixrowsep{0.5cm}
\xymatrixcolsep{0.1cm}\xymatrix @C=0.2em {
 \scriptstyle{(3,1)} \ar[rr]& &\scriptstyle{(3,2)}\ar[rd]   &  &  \\
  & \scriptstyle{(2,1)}  \ar[lu]\ar[rd]   & &\scriptstyle{(2,2)}   \\
    &     &\scriptstyle{(1,1)}\ar[ru]  &   \\
}
\end{tabular}
\end{center}

Let $l_{31}=l_{32}-1 $ and $T(R)$  any tableau in  $V_{\mathcal{C}}(T(L))$. Then all   tableaux in the subspace $U(T(R))$ satisfy one of the following equivalent relations whose graphs are:

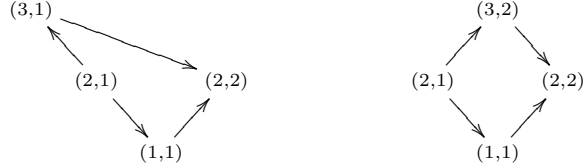
\begin{center}
\begin{tabular}{c c c}
\xymatrixrowsep{0.5cm}
\xymatrixcolsep{0.1cm}\xymatrix @C=0.2em {
 \scriptstyle{(3,1)} \ar[rrrd]& &   &  & \\
  & \scriptstyle{(2,1)}  \ar[lu]\ar[rd]   & &\scriptstyle{(2,2)}   \\
    &     &\scriptstyle{(1,1)}\ar[ru]  &   \\
}
\ \ \ \ \ \ \ \ \ \
\xymatrixrowsep{0.5cm}
\xymatrixcolsep{0.1cm}\xymatrix @C=0.2em {
  & &\scriptstyle{(3,2)}\ar[rd]   & & \\
  & \scriptstyle{(2,1)}  \ar[ru]\ar[rd]   & &\scriptstyle{(2,2)}   \\
    &     &\scriptstyle{(1,1)}\ar[ru]  &   \\
}
\end{tabular}
\end{center}
}
\end{example}

\begin{definition}
Let $\mathcal{C}$ be any set of relations. We call $(k,i)\in \mathfrak{V}(\mathcal{C})$
maximal if there exist no $(s,t)\in \mathfrak{V}(\mathcal{C})$ such that $(s,t)\succeq_{\mathcal{C}}(k,i)$. The minimal pair can be defined similarly.
Let $(i,j)\in \mathfrak{V}(\mathcal{C})$ be a maximal or a minimal pair. Denote by $\mathcal{C}_{ij}$ the set of relations obtained from $\mathcal{C}$ by removing all relations that involve $(i,j)$. 
\end{definition}

\begin{lemma}\label{lemma maximal}
For any set of relations $\mathcal{C}$ which does not contain loops, there exist maximal and minimal elements in $\mathfrak{V}(\mathcal{C})$.
\end{lemma}

\begin{proof}
Choose any $(i,j)\in \mathfrak{V}(\mathcal{C})$. If it is not an end point of an arrow then it is maximal.
If it is an ending vertex of an arrow then there exists $(r, s)\in \mathfrak{V}(\mathcal{C})$ such that
$((r, s),(i, j))\in \mathcal{C} $. Since $\mathcal{C}$  does not contain loops and $\mathfrak{V}(\mathcal{C})$ has finitely many vertices we  obtain a maximal element in $\mathfrak{V}(\mathcal{C})$ by repeating the procedure above. Existence of a minimum element can be proved similarly.
\end{proof}

\begin{lemma}\label{lemma loop}
Let $\mathcal{C}$ be a indecomposable subset of $\mathcal{R}$. There exists a tableau satisfying $\mathcal{C}$ if and only if $\mathcal{C}$ does not contain nontrivial loops.
\end{lemma}

\begin{proof}
Suppose $\mathcal{C}$ does not contain nontrivial loops. Then it contain a maximal element $(i_1,j_1)$. Let $ (r_1,s_1)\in \mathfrak{V}(\mathcal{C})$ be such that $(r_1,s_1)\neq(i_1,j_1)$ and $(r_1,s_1)\notin \mathfrak{V}(\mathcal{C}_{i_1j_1})$. If $(i_{t+1},j_{t+1})=(r_t,s_t)$, repeating this procedure we get a sequence
$(i_1, j_1),(i_2, j_2),\ldots,(i_m,j_m)$, where $m=\sharp \mathfrak{V}(\mathcal{C})$.
Let $T(L)$ be the tableau such that $l_{i_1, j_1}>l_{i_2, j_2}>\cdots>l_{i_m, j_m}$. Then it satisfies $\mathcal{C}$.

Suppose now that $\mathcal{C}$ contains a nontrivial loop and there is a tableau satisfying $\mathcal{C}$. Then
$l_{i_1, j_1}\geq l_{i_2, j_2}\geq \cdots \geq l_{i_m,j_m}$ with $(i_1,j_1)=(i_m,j_m)$ and there is at least one of the $\geq$ is $>$. Thus
$l_{i_1, j_1}> l_{i_1, j_1}$ which is a contradiction.
\end{proof}

So from now on, we will only consider $\mathcal{C}$ which do not contain loops.

Now we describe an effective method of constructing of admissible subsets of relations which we call {\it relations removal method} ({\it RR-method} for short).

We say that $\widetilde{\mathcal{C}}\subsetneq\mathcal{C}$ is obtained from $\mathcal{C}$ by the RR-method if it is obtained by a sequence removing of relations of the form $\mathcal{C'}\to \mathcal{C'}_{ij}$ for different indexes.

Let $T(L)$ be any $\mathcal{C}$-realization.
Then the Gelfand-Tsetlin formulas \eqref{Gelfand-Tsetlin formulas} define an action of $\mathfrak{Gl}_n$ on $V_{\mathcal{C}}(T(L))$.

\begin{theorem}\label{tech}
Let $ \mathcal{C}_1\subseteq \mathcal{R}$ be admissible and suppose that $\mathcal{C}_2$ is obtained from $ \mathcal{C}_1$ by the RR-method.  Then $\mathcal{C}_2$ is admissible.
\end{theorem}

\begin{proof}

Clearly, it suffices to prove the theorem in the case when $\mathcal{C}_2=(\mathcal{C}_1)_{ij}$ for some maximal or minimal $(i,j)\in \mathfrak{V}(\mathcal{C}_1)$.
To show that $\mathcal{C}_2$ is admissible it is sufficient to prove that for any $\mathcal{C}_2$-realization
$T(L)$ and any  generator $g$ of $\mathfrak{k}_n$ we have $gT(L +z)=0$, where $z\in {\mathbb Z}_0^{\frac{n(n+1)}{2}}$ is such that
$T(L + z )\in {\mathcal B}_{\mathcal{C}_{2}}(T(L))$.

Let $T(R)$ be a $\mathcal{C}_1$-realization such that $r_{st}=(l+z)_{st}$ if $(s, t)\neq (i, j)$.
Let $m$ be positive (respectively, negative) integer if $(i,j)$ is maximal (respectively minimal) with $|m|>3$.
Then $T(R+ m\boldsymbol{\delta}^{ij})$  satisfies $ \mathcal{C}_1$ and the
Gelfand-Tsetlin formulas \eqref{Gelfand-Tsetlin formulas} define a $\gl_n$-module structure on
$V_{\mathcal{C}_{1}}(T(R +m\boldsymbol{\delta}^{ij}))$.
 We have
\begin{align*}
gT(R+ m\boldsymbol{\delta}^{ij})&=\sum\limits_{w\in A}g_{w}(R+  m\boldsymbol{\delta}^{ij})T(R+ m\boldsymbol{\delta}^{ij} + w),
\end{align*}
where $A$ is the set of $w\in {\mathbb Z}_0^{\frac{n(n+1)}{2}}$ such that $T(R + m\boldsymbol{\delta}^{ij} +w )\in {\mathcal B}_{\mathcal{C}_{1}}(T(R)$.
Expanding
$gT(R+ m\boldsymbol{\delta}^{ij})$ and $gT(L+z)$  step by step
we have that
in every step the tableaux $T(L  +z+w )$ appearing in the expansion of $gT(L+z)$ are in ${\mathcal B}_{\mathcal C_2}(T(L))$
if and only if any $T(R+ m\boldsymbol{\delta}^{ij}+w)$ appearing in the expansion of $gT(R+ m\boldsymbol{\delta}^{ij})$  is in ${\mathcal B}_{\mathcal C_2}(T(L))$.
Thus,
\begin{align*}
gT(L+z )&=\sum_{w \in A }g_{w}(L+z )T(L+z +w).
\end{align*}
Since $V_{\mathcal{C}_1}(T(R+m\boldsymbol{\delta}^{ij}))$ is a module for infinitely many values of $m$ and
$g_{w}(R+  m\boldsymbol{\delta}^{ij})$  are rational functions in  $m$,
we conclude that $g_{w}(L+z)=0$ for all $w$ and  $ \mathcal{C}_2 $  is admissible.

\end{proof}

 A tableau $T(L)$ is called generic if $l_{ki}-l_{kj}\notin\mathbb{Z}$ for any $1\leq i\neq j\leq k\leq n-1$.
 The following corollary is a consequence of Theorem \ref{Gelfand-Tsetlin theorem} and Theorem \ref{tech}.

\begin{corollary}\label{cor generic}
Let $T(L)$ be a generic  Gelfand-Tsetlin tableau of height $n$. Denote by ${\mathcal B}(T(L))$  the set of all Gelfand-Tsetlin tableaux $T(R)$ satisfying $r_{nj}=l_{nj}$, $r_{ij}-l_{ij}\in\mathbb{Z}$ for  $1\leq j\leq i \leq n-1$. Then $V(T(L))=\Span {\mathcal B}(T(L))$ has a structure of a $\mathfrak{gl}_{n}$-module with the action of $\gl_n$ given by the Gelfand-Tsetlin formulas.
\end{corollary}

\begin{proof}
Applying the RR-method to $\mathcal{S}$ (see Example \ref{standard and generic relations}), after finitely many steps we can remove all the relations in $\mathcal{S}$. It follows from Theorems \ref{Gelfand-Tsetlin theorem} and \ref{tech} that $\emptyset$ is admissible. Since any generic tableau is an $\emptyset$-realization, the statement is proved.
\end{proof}

Since $\mathcal{S} $ is admissible we immediately obtain a large family of admissible sets of relations by Theorem \ref{tech}.
The following are  the graphs associated with some admissible sets of relations:

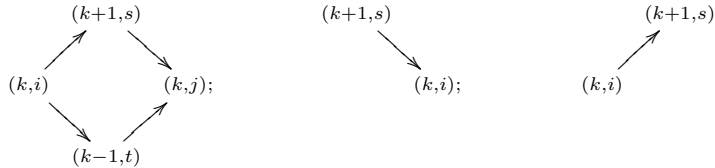
\begin{center}
\begin{tabular}{ c c c c c}
\xymatrixrowsep{0.5cm}
\xymatrixcolsep{0.1cm} \xymatrix @C=0.2em{
     &\scriptstyle{(k+1,s)}\ar[rd]   &   & \\
   \scriptstyle{(k,i)}\ar[rd] \ar[ru]  &    &\scriptstyle{(k,j)};   &  \\
      &\scriptstyle{(k-1,t)}\ar[ru]   &   & }
&\ &\ \xymatrixrowsep{0.5cm}
\xymatrixcolsep{0.1cm}\xymatrix @C=0.2em {
\scriptstyle{(k+1,s)}\ar[rd]  &   &       & \\
   &\scriptstyle{(k,i)}; &    &    \\
      &   &   & }& &
\xymatrixrowsep{0.5cm}
\xymatrixcolsep{0.1cm}\xymatrix @C=0.2em {
     &\scriptstyle{(k+1,s)}    &   & \\
   \scriptstyle{(k,i)}  \ar[ru]  &    &   &  \\
      &   &   & }
\end{tabular}
\end{center}

 Note that any subset of $\mathcal{R}$ is  a union of disconnected  indecomposable sets and this decomposition is unique.
Lemma \ref{lemma loop} gives a necessary condition for $\mathcal{C}$ to be admissible. In the following sections
we will give  necessary and sufficient conditions for admissibility.

\subsection{Noncritical sets of relations}

 We call a tableau $T(L)$ {\it noncritical} if $l_{ki}\neq l_{kj}$ for all $1\leq i< j\leq k\leq n-1$,  and {\it critical} otherwise. If $T(R)$ is critical then the Gelfand-Tsetlin formulas are not defined on $T(R)$ and, hence, $V_{\mathcal{C}}(T(L))$ should not contain critical tableaux.
So as the first step in construction of admissible relations we describe noncritical sets of relations.

\begin{definition}
Let $\mathcal{C}, \mathcal{C}'$ be any subsets of $\mathcal{R}$. We say that  $\mathcal{C}$  {\it implies} $\mathcal{C}'$ if whenever we have $(i,j)\succeq_{\mathcal{C}'} (r,s)$ we also have $(i,j)\succeq_{\mathcal{C}} (r,s)$. We say that $\mathcal{C}$ is equivalent to $\mathcal{C}'$ if $\mathcal{C}$ implies $\mathcal{C}'$ and $\mathcal{C}'$ implies $\mathcal{C}$.
\end{definition}

The following is easy to verify.
\begin{lemma}\label{lemma critical}
 Let $\mathcal{C}$ be a set of relations and $T(R)$ a tableau satisfying $\mathcal{C}$. Let $(s,t)$ be maximal or a minimal with respect to $\mathcal{C}$ and $(s,t)\neq(k,i),(k,j)$.
 There exists a tableau $T(Q)=T(R+z)$ satisfying $\mathcal{C}$ such that $q_{ki}= q_{kj}$
if and only if there exists a tableau $T(Q')=T(R+z')$ satisfying $\mathcal{C}_{st}$ such that $q'_{ki}= q'_{kj}$.
\end{lemma}

\begin{proof}
We give the proof when $(s,t)$ is maximal. Suppose that there exists a tableau $T(Q')=T(R+z')$ satisfying $\mathcal{C}_{st}$ such that $q'_{ki}= q'_{kj}$. Let $T(Q)=T(Q'+m\boldsymbol{\delta}^{ij})$ with $m\gg 0$, then $T(Q)$ satisfies  $\mathcal{C}$ and $q_{ki}= q_{kj}$.
The converse is obvious.
\end{proof}

\begin{proposition}\label{corollary bad set}
Let $\mathcal{C}$ be any subset of $\mathcal{R}$. If there is a critical tableau $T(L)$ satisfying $\mathcal{C}$
such that
$l_{ki}= l_{kj}$, for some $(k,i),(k,j)\in \mathfrak{V}(\mathcal{C})$, $1\leq i\neq j \leq k\leq n-1 $,
then for any $T(R)$ satisfying $\mathcal{C}$ there exists $z\in \mathbb Z_0^{\frac{n(n+1)}{2}}$ such that $T(R+z)=(q_{rs})$  with $q_{ki}= q_{kj}$.
\end{proposition}

\begin{proof}
If $(s,t)$ is maximal or minimal and $(s,t)\neq(k,i),(k,j)$,
we remove  $(s,t)$.
Repeating above removal method until there is not maximal or minimal element $(s,t)\neq(k,i),(k,j)$, we obtain
$\mathcal{C}'$. We claim that $\mathcal{C}'=\emptyset$, then $T(R+(r_{ki}-r_{kj})\boldsymbol{\delta}^{kj})$ satisfies $\emptyset$.
The statement follows from Lemma \ref{lemma critical}.

Assume  $\mathcal{C}'\neq\emptyset$. Since no subset of $\mathcal{C}$ contain loops then
there exist maximal and minimal elements in $\mathfrak{V}(\mathcal{C}')$. There are two possibilities:

(1) $(k,i)$ is maximal and $(k,j)$ is minimal,

(2) $(k,i)$ is  minimal and $(k,j)$ is maximal.

Both are impossible. For case (1) we have $(k,i)\succeq_{\mathcal{C}}(k,j)$ which implies $l_{ki}>l_{kj}$, since $T(L)$ satisfies $\mathcal{C}$. This contradicts with $l_{ki}=l_{kj}$. The case (2) can be proved to be impossible similarly.
\end{proof}

Proposition \ref{corollary bad set} suggests the following definition.

\begin{definition}
Let $\mathcal{C}$ be any subset of $\mathcal{R}$. We call $\mathcal{C}$ noncritical if for any $\mathcal{C}$-realization $T(L)$, one has $l_{ki}\neq l_{kj}$, $1\leq k\leq n-1,\ i\neq j$, $(k,i),(k,j)\in \mathfrak{V}(\mathcal{C})$.
\end{definition}

\begin{remark}\item[(i)] For any noncritical set $\mathcal{C}$,  there are infinitely many $\mathcal{C}$-realizations.
\item[(ii)] $\mathcal{S}$ is noncritical. Any standard tableau is a $\mathcal{S}$-realization.
\item[(iii)] $\emptyset$ is noncritical.  Any generic tableau is a $\emptyset$-realization.
\item[(iv)]  $\mathcal{C}=\{((2,1);(1,1)), ((2,2);(1,1))\}$ is critical since there exist tableaux $T(L)$ satisfying $\mathcal{C}$ with $l_{21}=l_{22}$.
\end{remark}


From now on we only consider  sets  of relations $\mathcal{C}$ satisfying the following conditions:
\begin{itemize}
\item[(i)] $\mathcal{C} $ does not contain loops.
\item[(ii)] $\mathcal{C} $ is noncritical.
\item[(iii)] If $(n,i)$ and $(n,j)$, $i\neq j$, are in the same indecomposable subset of $\mathcal{C}$,
 then $(n,i)\succeq_{\mathcal{C}}(n,j)$ or $(n,j)\succeq_{\mathcal{C}}(n,i)$.
\end{itemize}

\subsection{Reduced sets}

Our next step towards constructing the admissible sets of relations is to define reduced sets.

We have

\begin{proposition}\label{order}
Let $\mathcal{C}$ be an indecomposable noncritical set.
For any $(k,i),(k,j)\in \mathfrak{V}(\mathcal{C})$, $i\neq j$, we have that
$(k,i)\succeq_{\mathcal{C}}(k,j)$ or $(k,j)\succeq_{\mathcal{C}}(k,i)$.

\end{proposition}

\begin{proof}
Suppose there exist two tableaux $T(L)$ and $T(R)$ satisfying $\mathcal{C}$ with $l_{ki}>l_{kj}$ and $r_{ki}<r_{kj}$, $1\leq k\leq n-1$.
For any positive integers $s$ and $t$, the tableau $T(Q)=T(sL+tR)$ with entries $q_{ij}=sl_{ij}+tr_{ij}$, satisfies $\mathcal{C}$. In particular, for $s=r_{kj}-r_{ki}$ and $t=l_{ki}-l_{kj}$ one has  $q_{ki}=q_{kj}$ which is a contradiction. When $k=n$, the statement follows from condition (iii).
\end{proof}

Note that
any standard tableau $T(L)$ satisfies $l_{k1}>\cdots >l_{kk}$ for any $k$.

\begin{proposition}\label{choose tableau}
Let $\mathcal{C}$ be any noncritical set and  $m$ be any positive integer.  There exists a tableau $T(L)$ satisfying $\mathcal{C}$ such that $|l_{ij}-l_{st}|\geq m$ for
any two pairs $(i,j), (s,t)\in \mathfrak{V}(\mathcal{C})$.
\end{proposition}

\begin{proof}
 We use induction on $\#\mathfrak{V}(\mathcal{C})$.
The case $\#\mathfrak{V}(\mathcal{C})=2$ is clear. Suppose that $\#\mathfrak{V}(\mathcal{C})>2$ and let
$(i,j)$ be maximal. By induction there exists $T(L')$ satisfying $\mathcal{C}_{ij}$ and the condition of the proposition.
Let $T(L)= T(L'+s\boldsymbol{\delta}^{ij})$. If  we choose $s>>0$ then it satisfies $\mathcal{C}$ and the required condition.
\end{proof}


\begin{definition}\label{def reduced}
Let $\mathcal{C}$ be any noncritical set of relations. We call  $\mathcal{C}$ reduced, if for every $(k,j)\in \mathfrak{V}(\mathcal{C})$ the following conditions
are satisfied:
\begin{itemize}
\item[(i)] There exists at most one $i$ such that $((k,j);(k+1,i))\in \mathcal{C}$.
\item[(ii)] There exists at most one $i$ such that $((k+1,i);(k,j))\in\mathcal{C}$.
\item[(iii)] There exists at most one $i$ such that $((k,j);(k-1,i))\in\mathcal{C}$.
\item[(iv)] There exists at most one $i$ such that $((k-1,i);(k,j))\in \mathcal{C}$.
\item[(v)] No relations in the top row follow from other relations.

\end{itemize}
\end{definition}

We will show now that distinct reduced sets of relations are nonequivalent.

\begin{proposition}\label{set equiv}
Let $\mathcal{C}$ and $\mathcal{C}'$ be reduced sets. Then $\mathcal{C}$ and $\mathcal{C}'$
are equivalent if and only if $\mathcal{C}=\mathcal{C}'$.
\end{proposition}

\begin{proof} Suppose that $\mathcal{C}$ and $\mathcal{C}'$ are equivalent. Then $\mathfrak{V}(\mathcal {C})=\mathfrak{V}(\mathcal {C}')$. Let $m=\#\mathfrak{V}(\mathcal {C})=\#\mathfrak{V}(\mathcal {C}')$.
We prove the statement by induction on $m$. The case $m\leq3$  is obvious.
By Lemma \ref{lemma maximal}, 
$\mathcal{C}$  contains a maximal element, say $(i,j )$. Since $\mathcal{C}$ and $\mathcal{C}'$ are equivalent, we have that $(i, j)$ is maximal in $\mathcal{C}'$.
Then $\mathcal{C}_{ij}$  and $\mathcal{C}'_{ij}$
are equivalent. If not, there exists $T(R)$ that satisfies only one of $\mathcal{C}_{ij}$, $\mathcal{C}'_{ij}$. Then $T(R+s\boldsymbol{\delta}^{ij})$ satisfies only one of $\mathcal{C}$, $\mathcal{C}'$ for certain $s$ which contradicts
to the equivalence of  $\mathcal{C}_1$ and $\mathcal{C}_2$.
By induction we have that
$\mathcal{C}_{ij}=\mathcal{C}'_{ij}$.

Now we show that $\mathcal{C}=\mathcal{C}'$.
Assume the contrary.
Let $\mathcal{D}$ (resp. $\mathcal{D}'$) be the set of relations in $\mathcal{C}$ (resp. $\mathcal{C}'$) which involve $(i,j)$. Then $\mathcal{C}=\mathcal{C}_{ij}\cup\mathcal{D}$ and $\mathcal{C}'=\mathcal{C}'_{ij}\cup\mathcal{D}'$. Without loss of generality, we assume that $((i,j);(a_1,b_1))\in \mathcal{D}$
and $((i,j);(a_1,b_1))\notin \mathcal{D}'$.
By equivalence of $\mathcal{C}$ and $\mathcal{C}'$, we have $(i,j)\succeq_{\mathcal{C}'}(a_1,b_1)$. Then there exists $(a_2,b_2)\neq(a_1,b_1)$ such that $((i,j);(a_2,b_2))\in \mathcal{D}$ and
$(a_2,b_2)\succeq_{\mathcal{C}_{ij}}(a_1,b_1)$.
If $((i,j);(a_2,b_2))\in \mathcal{D}$, then
$(i,j)\succeq_{\mathcal{C}}(a_1,b_1)$ follows from other relations which contradicts $\mathcal{C}_1$ being reduced.  Thus $((i,j);(a_2,b_2))\notin \mathcal{D}$.
By above argument we have $((i,j);(a_3,b_3))\in \mathcal{D}$ and $(a_3,b_3)\succeq_{\mathcal{C}_{ij}}(a_2,b_2)$.
We claim that $(a_3,b_3)\neq(a_1,b_1)$. Otherwise there is a loop in $\mathcal{C}_1$.
Repeating above procedure we have $((i,j);(a_4,b_4))\in \mathcal{D}'$, then
$((i,j);(a_5,b_5))\in \mathcal{D}$ with $(a_p,b_p)\neq(a_q,b_q)$, $1\leq p\neq q\leq 5$.
This contradicts with $\mathcal{C}$ being reduced. Therefore $\mathcal{D}=\mathcal{D}'$.
Thus $\mathcal{C}=\mathcal{C}'$.

The converse is obvious.
\end{proof}

We can prove now

\begin{theorem}\label{equiv to reduced set}
Any noncritical set of relations is equivalent to a unique reduced set of relations.
\end{theorem}

\begin{proof}
It is sufficient to prove that any indecomposable set is equivalent to a reduced set. Let $\mathcal{C}$ be indecomposable and $((k+1,i); (k,j)),\ ((k+1,i'); (k,j))\in\mathcal{C}$. By Proposition \ref{order} we have
$(k+1,i)\succeq_{\mathcal{C}}(k+1,i')$ or $(k+1,i')\succeq_{\mathcal{C}}(k+1,i)$.
Suppose first $(k+1,i)\succeq_{\mathcal{C}}(k+1,i')$. If we show that any tableau $T(R)$ satisfying $\mathcal{C}\setminus \{((k+1,i);(k,j))\}$ satisfies also $r_{k+1,i}\geq r_{k,j}$, then $\mathcal{C}\setminus \{((k+1,i); (k,j))\}$ implies $\mathcal{C}$.
Assume that we have a tableau $T(R)$ satisfying $\mathcal{C}\setminus \{((k+1,i); (k,j))\}$ with
$r_{k+1,i}< r_{k,j}$. Then $r_{k+1,i}<r_{k+1,i'}$. Take a tableau $T(L)$  satisfying $\mathcal{C}$. Hence its entries satisfy $l_{k,j}\leq l_{k+1,i'}<l_{k+1,i}$.
For any positive integers $s$ and $t$, $T(Q)=T(sL+tR)$ satisfies $\mathcal{C}\setminus \{((k+1,i); (k,j))\}$. In particular, if $s=r_{k+1,i'}-r_{k+1,i}$ and $t=l_{k+1,i}-l_{k+1,i'}$  then $q_{k+1,i}=q_{k+1,i'}$. Since $T(Q) $ satisfies the noncritical set $\mathcal{C}$ we obtain a contradiction.
 Thus we have $r_{k+1,i}\geq r_{k,j}$.
 If $(k+1,i')\succeq_{\mathcal{C}}(k+1,i)$, then using the same arguments one has that $\mathcal{C}\setminus \{ ((k+1,i'); (k,j))\}$ implies $\mathcal{C}$.

All other cases can be proved similarly. Thus any noncritical set is equivalent to a reduced one. The uniqueness follows from
 Proposition \ref{set equiv}.
\end{proof}

Recall that $G$ denotes the group $S_{n}\times\cdots\times S_{1}$.
\begin{definition}
 For any $\sigma=(\sigma[n],\sigma[n-1],\ldots,\sigma[2],\sigma[1])\in G$ and $\mathcal{C}\subseteq\mathcal{R}$ denote by $\sigma(\mathcal{C})$
  the set of relations:
$$\{((i,\sigma[i](j));(r,\sigma[r](s)))\ |\ ((i,j);(r,s))\in\mathcal{C}\}.$$
\end{definition}

If $\mathcal{C}$ is any noncritical reduced subset of $\mathcal{R} $ and $V_{\mathcal{C}}(T(L))$ is a $\gl_n$-module then $V_{\mathcal{C}}(T(L)) \backsimeq V_{\mathcal{\sigma C}}(T(\sigma L))$. So it is sufficient to consider the noncritical reduced sets  that satisfy the following condition: $(k,i)\succeq_{\mathcal{C}}(k,j)$, only if $i <j$.

\subsection{Cross elimination.}

Here we give two examples of non admissible sets.

\begin{example}\label{not admissible}
Set $k\leq n-1$. The sets of relations corresponding with the following graphs are not admissible.
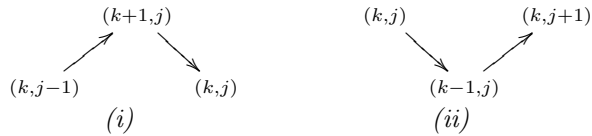
\begin{center}
\begin{tabular}{c c}
\xymatrixrowsep{0.5cm}
\xymatrixcolsep{0.1cm}\xymatrix @C=0.2em {
 &\scriptstyle{(k+1,j)}\ar[rd]   &   \\
 \scriptstyle{(k,j-1)}  \ar[ru]  &    &\scriptstyle{(k,j)}   \\
}

&\qquad \ \ \xymatrixrowsep{0.5cm}
\xymatrixcolsep{0.1cm}\xymatrix@C=0.2em{
 \scriptstyle{(k,j)}\ar[rd]   &    &\scriptstyle{(k,j+1)}   \\
 &\scriptstyle{(k-1,j)}\ar[ru]   &   \\
}\\
(i)  &\ \ \  (ii)
\end{tabular}
\end{center}
\end{example}

\begin{proof}
Suppose $\mathcal{C}_1$ is the set of relations of the diagram (i). Let $T(L)$  be any tableau
satisfying $\mathcal{C}_1$ and $T(R)$ be a tableau in $V_{\mathcal{C}_1}(T(L))$ such that $r_{k,j-1}=r_{k+1,j}-1$, $r_{k,j}=r_{k+1,j}$. By direct computation one has that
$e_kf_kT(R)-f_ke_kT(R)\neq (H_k-H_{k+1})T(R)$. Thus $V_{\mathcal{C}_1}(T(L))$ is not a $\gl_n$-module. Suppose $\mathcal{C}_{2}$ is the set of relations of the diagram (ii). Similarly we can prove that
$V_{\mathcal{C}_2}(T(L))$ is not a module for any $T(L)$ satisfying $\mathcal{C}_2$.
\end{proof}

\begin{definition}
Let $\mathcal{C}$ be an indecomposable noncritical subset of $\mathcal{R} $. A subset of $\mathcal{C}$ of the form { $\{((k,i); (k+1,t)),\ ((k+1,s);(k,j))\}$   }with $i<j$ and $s<t$ will be called a \it{cross}.
The associated graph is as follows

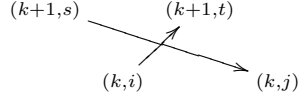
\begin{center}
\begin{tabular}{c}
\xymatrixrowsep{0.5cm}
\xymatrixcolsep{0.1cm}\xymatrix @C=0.2em {
\scriptstyle{(k+1,s)}\ar[rrrd]  &   &\scriptstyle{(k+1,t)}    &   \\
   &\scriptstyle{(k,i)}\ar[ru]  &    &\scriptstyle{(k,j)}   }
\end{tabular}
\end{center}
\end{definition}

\begin{example}\label{ex cross}
Let $\mathcal{C}$ and $\mathcal{C}_2$ be a set of relations defined by the following graphs
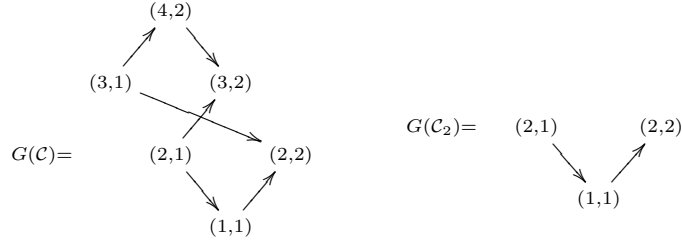
\begin{center}
\begin{tabular}{ c c c c }
\xymatrixrowsep{0.5cm}
\xymatrixcolsep{0.1cm}\xymatrix @C=0.01em {
& &\scriptstyle{(4,2)}\ar[rd]   & & &   & \\
  & \scriptstyle{(3,1)}  \ar[ru]\ar[rrrd]   & &\scriptstyle{(3,2)}& &  \\
\scriptstyle{G(\mathcal{C})=}    &     &\scriptstyle{(2,1)}\ar[ru] \ar[rd] & &\scriptstyle{(2,2)}  & \\
  &    & &\scriptstyle{(1,1)} \ar[ru] & &\\
}
&\
\xymatrixrowsep{0.5cm}
\xymatrixcolsep{0.1cm}\xymatrix @C=0.1em {
& &   & & &   & \\
  &       & && & \\
\scriptstyle{G(\mathcal{C}_2)=}    &     &\scriptstyle{(2,1)}  \ar[rd] & &\scriptstyle{(2,2)}  & \\
  &    & &\scriptstyle{(1,1)} \ar[ru] & &\\
}

\end{tabular}
\end{center}

Note that $(3,1)$ is maximal in $G(\mathcal{C})$, we get $\mathcal{C}_1$ after removing $(3,1)$ from $\mathcal{C}$. Then we can use the RR-method to remove the minimal element $(3,2)$ from $G(\mathcal{C}_1)$. This  gives $G(\mathcal{C}_2)$ which is not admissible by Example \ref{not admissible}. Thus $\mathcal{C}$  is not admissible by Theorem \ref{tech}.

\end{example}

Using the idea of Example \ref{ex cross} we can prove the following proposition.

\begin{proposition}\label{prop-cros}
Let $\mathcal{C}$ be an indecomposable noncritical subset of $\mathcal{R} $.
If $\mathcal{C}$ contains crosses then it is not admissible.
\end{proposition}

\begin{proof}
Suppose  $\mathcal{C}$ contains a cross { $\{((k,i); (k+1,t)),\ ((k+1,s);(k,j))\}$   }.
If $(s,t)$ is maximal or minimal and $(s,t)\notin\{(k,i),(k,j),(k+1,s),(k+1,t)\}$,
we remove  $(s,t)$. Repeating the above removal method until there is no such maximal or minimal pair different from $(k,i),(k,j),(k+1,s),(k+1,t)$, we obtain
$\mathcal{C}'$.
By Theorem \ref{tech} it is sufficient to prove that $\mathcal{C}'$ is not admissible.
The set
$\mathcal{C}'$ is not empty since  $\{((k,i);(k+1,t)),\ ((k+1,s);(k,j))\}$
is contained in $\mathcal{C}'$. We claim that $(k,i)$ and $(k+1,s)$ are maximal elements,  and
$(k,j),(k+1,t)$ are  minimal elements in $\mathfrak{V}(\mathcal{C}')$.
Suppose $(k+1,s)$ is not maximal. As $\mathcal{C}$ does not contain loops, then $(k,i)\succeq_{\mathcal{C}'} (k+1,s)$.
Since $\mathcal{C}'$ is a subset of $\mathcal{C}$, we have  $(k,i)\succeq_{\mathcal{C}} (k+1,s)$.
It implies $(k,i)\succeq_{\mathcal{C}}  (k+1,t)$ together with $(k+1,s)\succeq_{\mathcal{C}} (k+1,t)$, which contradicts to the fact that  $\mathcal{C}$ is reduced. Similarly one can prove that $(k ,i)$ is maximal and  $(k,j), \ (k+1,t)$ are  minimal.

If there is a path from $(k,i)$ to $(k+1,t)$ except $((k,i);(k+1,t))$,
then $(k,i)\succeq_{\mathcal{C}}(k+1,t)$ is implied by the relations involved in the path, which contradicts with that $\mathcal{C}$ is reduced.
There is a path from $(k,i)$ to $(k,j)$ and $ (k+1,t),\ (k+1,s) $ are not involved in these relations.
After removing relations involving $ (k+1,s)$ and $(k+1,j)$ by RR-method we obtain a set of relations $\mathcal{C}''$. It is sufficient to prove that $\mathcal{C}''$ is not admissible.
 It is clear that $(k,i)\succeq_{\mathcal{C}} (k,j)$. Since $\mathcal{C}$ is reduced, there exists $(a,b)$ in $\mathfrak{V}(\mathcal{C}'')$ such that  $a<k$.
 Let $a$ be the minimal
 such that there exist $((a+1,c);(a,b))\in \mathcal{C}''$ and $((a,b);(a+1,d))\in \mathcal{C}''$.
 Then we apply the RR-method to remove the relation that does not involve
$(a,b),(a+1,c),(a+1,d)$. We obtain the set of relations $\{((a+1,c);(a,b)),\ ((a,b); (a+1,d))\}$ which is not admissible by Example \ref{not admissible}. The statement follows.
\end{proof}

\begin{corollary}
Let $\mathcal{C}$ be any noncritical subset of $\mathcal{R} $.
If an indecomposable subset of $\mathcal{C}$ contains crosses then $\mathcal{C}$ is not admissible.
\end{corollary}
\begin{proof}
Let $\mathcal{C}_1$ be a subset of $\mathcal{C}$ which contains all crosses. We apply the RR-method to other subsets of $\mathcal{C}$ until all other subsets become empty. Then we obtain $\mathcal{C}_1$ which is not admissible by Proposition \ref{prop-cros}. Hence $\mathcal{C}$ is not admissible by Theorem \ref{tech}.
\end{proof}

\begin{definition}\label{definition: preadmissible}
Let $\mathcal{C}$ be an indecomposable set. We say that $\mathcal{C} $ is {\it pre-admissible}  if it satisfies the following conditions:

\begin{itemize}
\item[(i)] $\mathcal{C} $ does not contain loops.
\item[(ii)] $\mathcal{C} $ is noncritical.
\item[(iii)] For any $1\leq k\leq n$, $(k, i)\succeq_{\mathcal{C}} (k, j)$ if and only if $(k,i), (k,j)$ are in the same indecomposable subset of   $\mathcal{C} $ and $i<j$.
\item[(iv)] $\mathcal{C} $ is reduced.
\item[(v)] There is not cross in $\mathcal{C} $.
\end{itemize}

An arbitrary set $\mathcal{C}$ is pre-admissible if every  indecomposable subset of $\mathcal{C}$ is pre-admissible.
\end{definition}
The results of the previous sections show that in order to construct $\gl_n$-modules using sets of relations it is enough to consider only pre-admissible sets of relations.

\subsection{$\mathfrak{F}$ set}

\begin{proposition}
Let $\mathcal{C}$ be an indecomposable pre-admissible set.
If $T(L)$ is a $\mathcal{C}$-realization and $l_{ki}-l_{kj}=1$, then
one of the listed conditions holds.
\begin{itemize}
\item[(i)] $\{((k,i); (k-1,s)),\ ((k-1,s);(k,j))\}\subseteq \mathcal{C}$ for some $1\leq s\leq k-1$.
\item[(ii)] $\{((k,i);(k+1,s)),\ ((k+1,s);(k,j))\}\subseteq \mathcal{C}$ for some $1\leq s\leq k+1$.
\item[(iii)] $k=n-1$,  $\{((n-1,i);(n,s)),\ ((n,t); (n-1,j))\}\subseteq \mathcal{C}$ for some $s\leq t$ and $(n,s) \succeq_{\mathcal{C}}(n,t)$.
\end{itemize}
\end{proposition}
\begin{proof}
By Proposition \ref{order} there is a path in $G(\mathcal{C})$ from $(k,i)$ to $(k,j)$. If the path
goes up more than once then $r_{ki}-r_{kj}>1$ for any $T(R)$ satisfying $\mathcal{C}$. It proves the statement.

\end{proof}



Let $\mathcal{C}$ be a set of relations. For $i<j$ we call $\{(k,i),\ (k,j)\}$ an adjoining pair if they are in the same indecomposable subset of $\mathcal{C}$ and there is not $(k,s)$ such that
$(k,i)\succeq_{\mathcal{C}}(k,s)\succeq_{\mathcal{C}}(k,j)$.
Now we  introduce our main sets of relations which lead to admissible sets.
Denote by $\mathfrak{F}$ the set of all indecomposable $\mathcal{C}$ satisfying the following condition:

\begin{condition}\label{condition for admissible}

For every adjoining pair $(k,i)$ and $(k,j)$, $1\leq k\leq n-1$, there exist $p, q$ such that $\mathcal{C}_{1}\subseteq\mathcal{C}$
or, there exist $s<t$ such that $\mathcal{C}_{2}\subseteq\mathcal{C}$, where the graphs associated to $\mathcal{C}_{1}$ and $\mathcal{C}_{2}$ are as follows

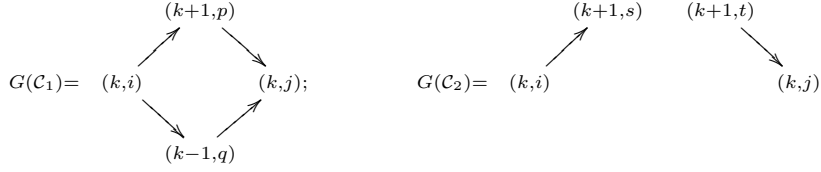
\begin{center}
\begin{tabular}{c c c c}
\xymatrixrowsep{0.5cm}
\xymatrixcolsep{0.1cm}
\xymatrix @C=0.2em{
  &   &\scriptstyle{(k+1,p)}\ar[rd]   &   & \\
 \scriptstyle{G(\mathcal{C}_{1})=}  &\scriptstyle{(k,i)}\ar[rd] \ar[ru]  &    &\scriptstyle{(k,j)};   &  \\
   &   &\scriptstyle{(k-1,q)}\ar[ru]   &   & }
&\ \ &
\xymatrixrowsep{0.5cm}
\xymatrixcolsep{0.1cm}\xymatrix @C=0.2em {
   &   &\scriptstyle{(k+1,s)}    &   &\scriptstyle{(k+1,t)}\ar[rd]&& \\
  \scriptstyle{G(\mathcal{C}_{2})=} &\scriptstyle{(k,i)} \ar[ru]  & &   & & \scriptstyle{(k,j)} \\
   &   &   &   & &&}
\end{tabular}
\end{center}
\end{condition}
We finish this section with some technical lemmas which will be used to prove that Condition \ref{condition for admissible} implies admissibility of sets of relations.

\begin{lemma}\label{lemma1}
Let $\mathcal{C}\in \mathfrak{F}$, $T(L)$  a $\mathcal{C}$-realization and $T(R)$ any tableau in  $\mathcal{B}_{\mathcal{C}}(T(L))$ such that $r_{ki}-r_{kj}=1$. If
$A_1:=\{i' \ |\ r_{k+1,i'}=r_{kj} \}$
and $A_2:=\{ j'\ |\  \ r_{k-1,j'}=r_{ki}\}$, then $\#A_1+\#A_2\geq2$.
\end{lemma}
\begin{proof}
Follows directly from $\mathcal{C}\in \mathfrak{F}$.
\end{proof}





Let  $W=(w_{i,j})\in\mathbb{C}^{\frac{n(n+1)}{2}}$. Set
\begin{equation}
e_{ki}(W)=\left\{
\begin{array}{cc}
0,& \text{ if } T(W)\notin  \mathcal {B}{_\mathcal{C}}(T(L))\\
-\frac{\prod\limits_{j=1}^{k+1}(w_{ki}-w_{k+1,j})}{\prod\limits_{j\neq i}^{k}(w_{ki}-w_{kj})},& \text{ if } T(W)\in  \mathcal {B}{_\mathcal{C}}(T(L))
\end{array}
\right.
\end{equation}

\begin{equation}
f_{ki}(W)=\left\{
\begin{array}{cc}
0,& \text{ if } T(W)\notin  \mathcal {B}_{\mathcal{C}}(T(L))\\
\frac{\prod\limits_{j=1}^{k-1}(w_{ki}-w_{k-1,j})}{\prod\limits_{j\neq i}^{k}(w_{ki}-w_{kj})},& \text{ if } T(W)\in  \mathcal {B}_{\mathcal{C}}(T(L))
\end{array}
\right.
\end{equation}

\begin{equation}
h_{k}(W)=\left\{
\begin{array}{cc}
0,& \text{ if } T(W)\notin  \mathcal {B}{_\mathcal{C}}(T(L))\\
 \sum\limits_{i=1}^{k}w_{ki}-\sum\limits_{i=1}^{k-1}w_{k-1,i}+k-1,& \text{ if } T(W)\in\mathcal {B}_{\mathcal{C}}(T(L))
\end{array}
\right.
\end{equation}

\begin{equation}
\Phi(W,z_1,\ldots,z_m)=
\left\{
\begin{array}{cc}
1,& \text{ if }T(W+z_1+\ldots+z_t)\in  \mathcal {B}_{\mathcal{C}}(T(L)) \text{ for any } t\\
0,& \text{ otherwise}.
\end{array}
\right.
\end{equation}

We will denote by $T(v)$ the tableau with variable entries $v_{ij}$.

\begin{lemma}\label{action of center1}

Let $\mathcal{C}\in\mathfrak{F}$, $T(L)$ any $\mathcal{C}$-realization.
\begin{itemize}
\item[(i)] If $T(L+\boldsymbol{\delta}^{kj})\notin \mathcal {B}{_\mathcal{C}}(T(L))$ and
$l_{k,i}-l_{kj}\neq 1$ for any $i$, then  $$\lim\limits_{v\rightarrow l}e_{kj}(v)f_{kj}(v+\boldsymbol{\delta}^{kj})=0.$$
\item[(ii)] If $T(L-\boldsymbol{\delta}^{kj})\notin \mathcal {B}{_\mathcal{C}}(T(L))$ and  $l_{k,j}-l_{k,i}\neq 1$ for any i,
 then $$\lim\limits_{v\rightarrow l}f_{kj}(v)e_{kj}(v-\boldsymbol{\delta}^{kj})=0.$$
\item[(iii)] If $l_{k,i}-l_{k,j}= 1$,
 then $T(L+\boldsymbol{\delta}^{k,j}), T(L-\boldsymbol{\delta}^{k,i})\notin \mathcal {B}{_\mathcal{C}}(T(L))$, and
$$\lim\limits_{v\rightarrow l}e_{kj}(v)f_{kj}(v+\boldsymbol{\delta}^{k,j})-f_{ki}(v)e_{ki}(v-\boldsymbol{\delta}^{k,i})=0.$$
\end{itemize}
\end{lemma}

\begin{proof}Since $T(L+\boldsymbol{\delta}^{kj})\notin \mathcal {B}{_\mathcal{C}}(T(L))$, we  have $((k+1,s);(k,j))\in\mathcal{C}$ or $((k-1,t);(k,j))\in \mathcal{C}$.
Suppose $((k+1,s);(k,j))\in\mathcal{C}$ and $T(L+\boldsymbol{\delta}^{kj})\notin \mathcal {B}{_\mathcal{C}}(T(L))$. Then $l_{k+1,s}=l_{k,j}$.
 Suppose $((k-1,t);(k,j))\in \mathcal{C}$ and $T(L+\boldsymbol{\delta}^{kj})\notin \mathcal {B}{_\mathcal{C}}(T(L))$. Then $l_{k-1,t}=l_{k,j}+1$
 In both cases one has  $\lim\limits_{v\rightarrow l}e_{kj}(v)f_{kj}(v+\boldsymbol{\delta}^{kj})=0$  by direct computation .

The proof of (ii) is similar to (i).

It is clear that $T(L-\boldsymbol{\delta}^{k,j}),\  T(L+\boldsymbol{\delta}^{k,j+1})\notin \mathcal {B}{_\mathcal{C}}(T(L))$ if $l_{k,j}-l_{k,j+1}= 1$. By
Lemma \ref{lemma1},   one has
$$\lim_{v\rightarrow l}e_{kj}(v)f_{kj}(v+\boldsymbol{\delta}^{k,j+1})-f_{kj}(v)e_{kj}(v-\boldsymbol{\delta}^{k,j})=0.$$
\end{proof}

If  $((i,j);(r,s)),((r,s);(i,j))\notin\mathcal{C}$, we say that there is no direct relation between $(i,j)$ and $(r,s)$.

\begin{lemma} \label{lemma phi1}
 Let $\mathcal{C}\in\mathfrak{F}$, $T(L)$ be a $\mathcal{C}$-realization. Suppose that there is no direct relation between $(i_1,j_1)$ and $(i_2,j_2)$.
 Then
     $T(R+\boldsymbol{\delta}^{i_1j_1}+\boldsymbol{\delta}^{i_2j_2}){ \in \mathcal{B}_\mathcal{C}(T(L))}$ if and only if $T(R+\boldsymbol{\delta}^{i_1j_1}){ \in \mathcal{B}_\mathcal{C}(T(L))}$ and $T(R+\boldsymbol{\delta}^{i_2j_2}) { \in \mathcal{B}_\mathcal{C}(T(L))}$.
\end{lemma}

\begin{proof}
Suppose $T(R+\boldsymbol{\delta}^{i_1j_1}+\boldsymbol{\delta}^{i_2j_2}){ \in \mathcal{B}_\mathcal{C}(T(L))}$. Then
$T(R+\boldsymbol{\delta}^{i_1j_1})$ satisfies all relations that do not involve $(i_1,j_1)$.
It also satisfies all relations between $(i_1,j_1)$ and $(r,s)$, such that $(r,s)\neq(i_2,j_2)$ but there is no direct relation between  $(i_1, j_1)$ and $(i_2, j_2)$. Then $T(R+\boldsymbol{\delta}^{i_1j_1}){ \in \mathcal{B}_\mathcal{C}(T(L))}$. Similarly one can show  $T(R+\boldsymbol{\delta}^{i_2j_2}) { \in \mathcal{B}_\mathcal{C}(T(L))}$ and the converse.
\end{proof}

\subsection{Necessary and sufficient conditions of admissibility
}

Now we  are ready to describe admissible sets of relations.

\begin{theorem}\label{sufficiency of admissible} { A pre-admissible set of relations} $\mathcal{C}$ is admissible if and only if
$\mathcal{C}$ is a union of disconnected sets from $\mathfrak{F}$.
\end{theorem}

\begin{proof} Let $\mathcal{C}$ be a  union of disconnected sets from $\mathfrak{F}$ and $T(L)$ any $\mathcal{C}$-realization.
In order to prove that $V_{\mathcal{C}}(T(L))$ is a
$\gl_n$-module defined by the Gelfand-Tsetlin formulas one needs to show that $gT(R)=0$
for any $T(R)\in {\mathcal B}_{\mathcal{C}}(T(L))$ and any generator $g$ of $\mathfrak{k}_n$.

First we show that $gT(R)=0$ for $g=[e_i,[e_i,e_j]] \, (|i-j|=1)$.


\begin{equation}\label{formula}
[e_i,[e_i,e_j]]T(R)=\sum_{r,s,t}g_{rst}(R)T(R+\boldsymbol{\delta}^{jr}+\boldsymbol{\delta}^{is}+\boldsymbol{\delta}^{it}),
\end{equation}
where $g_{rst}(R)$ are rational functions in $r_{ij}, 1\leq j\leq i\leq n$.

Now we consider the coefficients $g_{rst}(R)$ of tableaux $T(R+\boldsymbol{\delta}^{jr}+\boldsymbol{\delta}^{is}+\boldsymbol{\delta}^{it})\in{\mathcal B}_{\mathcal{C}}(T(L))$.

\item [(i)] Suppose $s=t$.

\begin{itemize}
\item[(a)] Suppose there is not direct relation between $(i,s)$ and $(j,r)$. By Lemma \ref{lemma phi1} $\Phi(R,\boldsymbol{\delta}^{jr},\boldsymbol{\delta}^{is})=\Phi(R,\boldsymbol{\delta}^{is},\boldsymbol{\delta}^{is})=\Phi(R,\boldsymbol{\delta}^{is},\boldsymbol{\delta}^{jr})=1$.
Then the coefficient of $T(R+\boldsymbol{\delta}^{jr}+2\boldsymbol{\delta}^{is})$ is the limit of the coefficient of $T(v+\boldsymbol{\delta}^{jr}+2\boldsymbol{\delta}^{is})$  in $[e_i,[e_i,e_j]]T(v)$ when $v\rightarrow R$ (here $T(v)$ again is a tableau with free variable entries). Thus
the coefficient of $T(R+\boldsymbol{\delta}^{jr}+2\boldsymbol{\delta}^{is})$ is zero.
\item[(b)] Suppose $((i,s);(j,r))\in \mathcal{C}$ or $((j,r)(i,s))\in \mathcal{C}$.
Assume $((i,s);(j,r))\in \mathcal{C}$. Denote $\mathcal{C}'=\{((i,s);(j,r))\}$

      Let $T(v)$ be the tableau with $v_{s't'}=l_{s't'}$ if $(s',t')=(i,s) \text{ or } (j,r)$, and
      free variable entries $v_{s't'}$ otherwise. Then $T(v)$ is a $\mathcal {C'}$-realization and $V_{\mathcal{C}'}(T(v))$ is a $\gl_n$-module.
        Let $z^{(1)},z^{(2)}\in \{\boldsymbol{\delta}^{jr},\boldsymbol{\delta}^{is}\}$. Then $\Phi(R,z^{(1)},z^{(2)})=\Phi(v,z^{(1)},z^{(2)})$ where $z^{(1)}=z^{(2)}$ only if  $z^{(1)}=z^{(2)}=\boldsymbol{\delta}^{is}$. Therefore the coefficient of $T(R+\boldsymbol{\delta}^{jr}+2\boldsymbol{\delta}^{is})$ is the limit of the coefficient of $T(v+\boldsymbol{\delta}^{jr}+2\boldsymbol{\delta}^{is})$  in $[e_i,[e_i,e_j]]T(v)$  when $v\rightarrow R$, hence, it is zero.
\end{itemize}
\item [(ii)] Suppose $s\neq t$. We have that there is not direct relation between $(i,s)$ and $(i,t)$.
\begin{itemize}
\item[(a)] Suppose there is not direct relation between $(j,r)$ and both $\{(i,s),(i,t)\}$. Then
$\Phi(R,z^{(1)},z^{(2)},z^{(3)})=1$ by Lemma \ref{lemma phi1}, where $(z^{(1)},z^{(2)},z^{(3)})$ is any permutation of $\{\boldsymbol{\delta}^{is},\boldsymbol{\delta}^{it},\boldsymbol{\delta}^{jr}\}$.
 Thus the coefficient of $T(R+\boldsymbol{\delta}^{jr}+\boldsymbol{\delta}^{is} +\boldsymbol{\delta}^{it})$ is zero similarly to (a) in (i).

\item[(b)] Suppose there is a direct relation between $(j,r)$ and one of $\{(i,s),(i,t)\}$. Similarly to (b) in (i), one has that the coefficient of $T(R+\boldsymbol{\delta}^{jr}+\boldsymbol{\delta}^{is} +\boldsymbol{\delta}^{it})$ is zero.

\item[(c)] Suppose there are direct relations between $(j,r)$ and both $\{(i,s), (i,t)\}$.
We give proof for $j=i+1$, when $j=i-1$ it can be proved similarly.
Without loss of generality we assume that $s<t$.
We have that $\{((i,s);(j,r)),((j,r);(i,t))\}\subseteq \mathcal{C}$.
Then $(i,s)$ and $(i,t)$ are adjoining pairs in $i$-th row. There exist $p$ such that
$\{((i,s);(i-1,p)),((i-1,p);(i,t))\}\subseteq \mathcal{C}$.
Let $\mathcal{C}'=\{((i,s);(j,r)),((j,r);(i,t)),((i,s);(i-1,p)),((i-1,p);(i,t))\}$.
It is admissible by Theorem \ref{tech}.
Let $T(v)$ be the tableau with $v_{s't'}=l_{s't'}$ if $(s',t')\in\{(i,s),(i,t),(j,r),(i-1,p)\}$, and
      free variable entries $v_{s't'}$ otherwise. Then $T(v)$ is a $\mathcal {C'}$-realization and $V_{\mathcal{C}'}(T(v))$ is a $\gl_n$-module.
$\Phi(R,z^{(1)},z^{(2)},z^{(3)})=\Phi(v,z^{(1)},z^{(2)},z^{(3)})$, where $(z^{(1)},z^{(2)},z^{(3)})$ is any permutation of $\{\boldsymbol{\delta}^{is},\boldsymbol{\delta}^{it},\boldsymbol{\delta}^{jr}\}$.
Then the coefficient of $T(R+\boldsymbol{\delta}^{jr}+\boldsymbol{\delta}^{is}+\boldsymbol{\delta}^{it})$ is the limit of the coefficient of $T(v+\boldsymbol{\delta}^{jr}+\boldsymbol{\delta}^{is}+\boldsymbol{\delta}^{it})$  in $[e_i,[e_i,e_j]]T(v)$ when $v\rightarrow R$. Thus it is zero.
\end{itemize}

The proof of $gT(R)=0$ for generators
\eqref {h_i relation},
\eqref{h with e_i relation},
\eqref{h with f_j relation},
\eqref{ee and ff relations}
and  $[f_i,[f_i,f_j]]$
 is similar.
In the following we show that $[e_i,f_j]T(R)=\delta_{ij}(H_i-H_j)T(R)$. We have
\begin{equation}
\begin{split}
    [e_i,f_j]T(R)
    =&\sum_{r=1}^{j}\sum_{s=1}^{i}\Phi(R,-\boldsymbol{\delta}^{jr})f_{jr}(R)e_{is}(R+\boldsymbol{\delta}^{jr})T(R-\boldsymbol{\delta}^{jr}+\boldsymbol{\delta}^{is})\\
-&\sum_{r=1}^{j}\sum_{s=1}^{i}\Phi(R,\boldsymbol{\delta}^{is})e_{is}(R)f_{jr}(R+\boldsymbol{\delta}^{is})T(R-\boldsymbol{\delta}^{jr}+\boldsymbol{\delta}^{is}).\\
\end{split}
\end{equation}

Now we consider the coefficients of tableaux $T(R-\boldsymbol{\delta}^{jr}+\boldsymbol{\delta}^{is})\in{\mathcal B}_{\mathcal{C}}(T(L)) $.
If  $(i,r)\neq (j,s)$, then  the coefficient of  $T(R-\boldsymbol{\delta}^{jr}+\boldsymbol{\delta}^{is})$ is zero which can be proved similarly to the case $g=[e_i,[e_i,e_j]] \, (|i-j|=1)$, hence,  $[e_i,f_j]T(R)=0$ if $i\neq j$.

Suppose $i=j=k$.  The coefficient of $T(R-\boldsymbol{\delta}^{ir}+\boldsymbol{\delta}^{is})$ is zero if $r\neq s$.

By { Lemma} \ref{action of center1},
the coefficient of $T(R)$ is
\begin{align*}
&\lim_{v\rightarrow l}\left(\sum_{r=1}^{k}\sum_{s=1}^{k}f_{kr}(v)e_{ks}(v+\boldsymbol{\delta}^{kt})
-\sum_{r=1}^{k}\sum_{s=1}^{k} e_{ks}(v)f_{kr}(v+\boldsymbol{\delta}^{ks})\right)\\
&=\lim_{v\rightarrow R}h_{k}(v)
=h_{k}(R).
\end{align*}
Hence $[e_i,f_j]T(R)=\delta_{ij}h_iT(R)$ and
 $\mathcal{C}$ is admissible.

Conversely, assume  that $\mathcal{C}$ is admissible. We will show that  $\mathcal{C}$ is a union of disconnected sets from $\mathfrak{F}$.
Suppose first that $\mathcal{C}$ is indecomposable. If it does not satisfy condition \eqref{condition for admissible},
then one can choose a $\mathcal{C}$-realization $T(R)$ such that $r_{ki}-r_{kj}=1$ and $r_{ks}-r_{kt}\neq 1$ if
$(s,t)\neq (i,j)$ for some $1\leq k\leq n-1$.
 Thus $[e_k,f_k]T(R)\neq (H_k-H_{k+1})T(R)$ and  $V_{\mathcal{C}}(T(L))$ is not a $\gl_n$-module.
Suppose $\mathcal{C}= \mathcal{C}_1\cup\cdots\cup\mathcal{C}_m$. Without loss of generality we may assume that
$\mathcal{C}_1$ does not satisfy condition \eqref{condition for admissible}. Applying RR-method to $\mathcal{C}$ we can get $\mathcal{C}_1$ which is not admissible. Thus $\mathcal{C}$ is not admissible by Theorem \ref{tech}.
\end{proof}

\section{Admissible Gelfand-Tsetlin modules}

From now on we will assume that $\mathcal{C}$ is an admissible  subset of $\mathcal{R}$ and consider the  $\gl_n$-module
$V_{\mathcal{C}}(T(L))$. It  is endowed with the action of $\gl_n$  by the Gelfand-Tsetlin formulas (\ref{Gelfand-Tsetlin formulas}).


We will analyze the action of the Gelfand-Tsetlin subalgebra  $\Gamma$ on modules  $V_{\mathcal{C}}(T(L))$.
First we show that the action of $\Gamma$ is
preserved by the RR-method. Namely, we have

\begin{lemma}\label{lem-center}
Let  $ \mathcal{C}_1$ and $ \mathcal{C}_2$ be admissible sets such that $\mathcal{C}_2$ is obtained from $ \mathcal{C}_1$ by the RR-method.
 Then  $\Gamma$ acts on   $V_{\mathcal{C}_1}(T(L))$   by
  \eqref{action of Gamma in finite dimensional modules} for any ${\mathcal{C}_1}$-realization  $T(L)$ if and only if 
it acts on   $V_{\mathcal{C}_2}(T(\tilde{L}))$   by
  \eqref{action of Gamma in finite dimensional modules} for any ${\mathcal{C}_2}$-realization  $T(\tilde{L})$.
\end{lemma}

\begin{proof}
We give the proof in the case when $\mathcal{C}_2$ is obtained from $ \mathcal{C}_1$ by removing the relations involving $(i,j)$ which is minimal. If $(i,j)$ is maximal the statement can be proved by the same argument.

Suppose $\Gamma$ acts on   $V_{\mathcal{C}_1}(T(L))$   by
  \eqref{action of Gamma in finite dimensional modules} for any ${\mathcal{C}_1}$-realization.
  Let $T(\tilde{L})$ be any ${\mathcal{C}_2}$-realization, then there exists a ${\mathcal{C}_1}$-realization $T(L)$ such that
 $\tilde{l}_{rs}=l_{rs}$ for any $(r,s)\neq (i,j)$.
 Let
 $T(L')=T(L-p\boldsymbol{\delta}^{ij})$, $p\in \mathbb Z_{\geq 0}$, be a tableau satisfying
$ \mathcal{C}_1$.
Then for all possible $m$ and $k$ the equality  $c_{mk}(T(L'))=\gamma_{mk}(l')T(L')$ holds for any $p\geq 0$.
When $p>n$, for any $T(\tilde{L}+z')$ appearing in the expansion of $c_{mk}(T(\tilde{L}))$  we have that $T(\tilde{L}+z')\in \mathcal{B}_{\mathcal{C}_1}(T(L))$ if and only if $T(L'+z')\in \mathcal{B}_{\mathcal{C}_2}(T(L'))$. Considering $\gamma_{mk}(l')$ as a rational function in $p$ we conclude that $c_{mk}(T(\tilde{L}))=\gamma_{mk}(l)T(\tilde{L})$.

Conversely, suppose $\Gamma$ acts on   $V_{\mathcal{C}_2}(T(\tilde{L}))$    by
  \eqref{action of Gamma in finite dimensional modules} for any ${\mathcal{C}_2}$-realization.
Let $T(L)$ be any ${\mathcal{C}_1}$-realization, then there exists a ${\mathcal{C}_2}$-realization $T(\tilde{L})$ such that
 $\tilde{l}_{rs}=l_{rs}$ for any $(r,s)\neq (i,j)$. We have three cases to consider. If $m<i$, we have  $c_{mk}(T(L))=\gamma_{mk}(l)T(L)$ since  $\tilde{l}_{rs}=l_{rs}$ for $r<i$. Suppose now that $m=i$. Then if there is not direct relation between $(i,j)$ and $(i-1,s)$ for any $1\leq s \leq i-1$.
  For any $T(L+z')$ appearing in the expansion of $c_{mk}(T(\tilde{L}))$  we have that $T(L+z')\in \mathcal{B}_{\mathcal{C}_1}(T(L))$ if and only if $T(\tilde{L}+z')\in \mathcal{B}_{\mathcal{C}_2}(T(\tilde{L}))$.

  If $((i-1,s);(i,j))\in \mathcal{C}$ for some $1\leq s\leq i-1$. Since $c_{mk}\in U(\gl_i)$ and the elements in $\gl_i$ do not change the $(r,s)$-th entry of the tableau,, $i\leq r \leq n$, in this case it is enough to prove the statement for $i=n$.
  Since $i=n$, $T(L)$ is a ${\mathcal{C}_2}$-realization. The vector space $W$ spanned by tableaux $T(L+z)$ with
  $(l+z)_{i-1,s}\leq (l+z)_{i,j}$ is a submodule of $V_{\mathcal{C}_2}(T(L))$, then  $V_{\mathcal{C}_1}(T(L))$ is a quotient of $V_{\mathcal{C}_2}(T(L))$ so we have that $c_{mk}(T(L))=\gamma_{mk}(l)T(L)$ in the module $V_{\mathcal{C}_1}(T(L))$. Finally, if $m>i$ then by the above argument we have that $c_{mk}(T(L-n\boldsymbol{\delta}^{ij}))=\gamma_{mk}(L-n\boldsymbol{\delta}^{ij})T(L-n\boldsymbol{\delta}^{ij})$.
Since $m>i$, we have $\gamma_{mk}(L-n\boldsymbol{\delta}^{ij})=\gamma_{mk}(L-n\boldsymbol{\delta}^{ij})$.
Consider the $\gl_m$-module $W$ generated by  $T(L-n\boldsymbol{\delta}^{ij})$, if $T(L)$ is in $W$ then $c_{mk}(T(L))=\gamma_{mk}(l)T(L)$.

  Assume the contrary, $T(L)\notin W$. Let $p$ be the minimal such that $T(L-(p-1)\boldsymbol{\delta}^{ij})\notin W$
  and  $T(L-t\boldsymbol{\delta}^{ij})\in W$ for $p\leq t \leq n$. Consider the equation
  $E_{i,i+1}T(L-p\boldsymbol{\delta}^{ij})= \sum_{t=1}^{i}e_{it}(L-p\boldsymbol{\delta}^{ij})T(L-p\boldsymbol{\delta}^{ij}+\boldsymbol{\delta}^{it})$.

Since $T(L-(p-1)\boldsymbol{\delta}^{ij})$ and $T(L-p\boldsymbol{\delta}^{ij})$  are in $\mathcal{B}_{\mathcal{C}_1}(T(L))$,
the coefficient of $T(L-(p-1)\boldsymbol{\delta}^{ij})$ is non zero. By 2), each tableau appearing on the right hand side
has a different eigenvalue corresponding to $Z_i$, the center of $U_{gl_i}$,
so $T(L-(p-1)\boldsymbol{\delta}^{ij})\in W$ which contradicts with the minimality of $p$.
Thus $T(L)\in W$ and $c_{mk}(T(L))=\gamma_{mk}(l)T(L)$.

\end{proof}



\begin{corollary}\label{lem center generic}
Let $T(L)$ be a generic  Gelfand-Tsetlin tableau of height $n$. Then
the action of
the generators of the Gelfand-Tsetlin subalgebra on $V_{\emptyset}(T(L))=\Span {\mathcal B}_{\emptyset}(T(L))$ is given by the  formula \eqref{action of Gamma in finite dimensional modules}.
\end{corollary}

\begin{proof}
Applying the RR-method to $\mathcal{S}$, after finitely many steps we can remove all the relations in $\mathcal{S}$.  The statement follows from Lemma \ref{lem-center}.
\end{proof}

We have

\begin{theorem}\label{thm action of gamma}
For any admissible $\mathcal{C}$ the module $V_{\mathcal{C}}(T(L))$ is a Gelfand-Tsetlin module with diagonalizable action of
the generators of the Gelfand-Tsetlin subalgebra  given by the  formula \eqref{action of Gamma in finite dimensional modules}.
\end{theorem}

\begin{proof} If $\mathcal{C}$
  is an arbitrary admissible set then applying the RR-method we get $\emptyset$ after finitely many steps. Lemmas \ref{lem-center} and \ref{lem center generic} imply the statement.
\end{proof}

We call $\mathcal{C}$  a  {\it maximal} set of relations for $T(L)$, if $T(L)$ satisfies $\mathcal{C}$  and $\mathcal{C}$ implies any set of relations $\mathcal{C}'$ satisfied by $T(L)$.

We  call  $V_{\mathcal{C}}(T(L))$ {\it admissible Gelfand-Tsetlin module} associated with the admissible set of relations $\mathcal{C}$. Note that
$V_{\mathcal{C}}(T(L))$ is infinite dimensional if
 $\mathcal{C}$ is not equivalent to $\mathcal{S} $.

\begin{lemma}\label{lemma separation}
Let $\sum_{i=1}^{r}a_{i}T(L_i)\in V_{\mathcal{C}}(T(L))$ with all $a_{i}$ non-zero.
Then   $T(L_i)\in V_{\mathcal{C}}(T(L))$ for all $1\leq i\leq r$.
\end{lemma}

\begin{proof}
It follows from Theorem \ref{thm action of gamma} that the Gelfand-Tsetlin subalgebra has different  spectrum on each $T(L_i)$.
Hence  $T(L_i)\in V_{\mathcal{C}}(T(L))$ for each $i$.
\end{proof}

\begin{lemma}\label{lemma sequences}
Let $\mathcal{C}$ be an admissible set of relations.
 If $T(L)$ and $T(R)=T(L+z)$ satisfy $\mathcal{C}$. Then there exist $\{(i_t,j_t)\}_{t=1,\ldots,s}\subseteq \mathfrak{V}$
such that for any $1\leq r\leq s$, $T(L+\sum_{t=1}^{r}\epsilon_t\boldsymbol{\delta}^{i_t,j_t})$ satisfies $\mathcal{C}$
and $T(L+\sum_{t=1}^{s}\epsilon_t\boldsymbol{\delta}^{i_t,j_t})=T(R)$,
where
$\epsilon_t= 1$ if $r_{i_t,j_t}-l_{i_t,j_t}\geq 0$ and
$\epsilon_t= - 1$ if $r_{i_t,j_t}-l_{i_t,j_t}< 0$.
\end{lemma}

\begin{proof}
We prove the statement by induction on $\#\mathfrak{V}(\mathcal{C} )$. It is obvious if $\#\mathfrak{V}(\mathcal{C})=2$. Assume $\#\mathfrak{V}(\mathcal{C} )=n>2$. Let $(i,j)$ be maximal
and consider $\mathcal{C}_{ij}$. By induction, there exist
 sequences $(i_t',j_t')$ $1\leq t \leq s$
such that for any $p\leq s$, $T(L+\sum_{t=1}^{p}\epsilon_t\boldsymbol{\delta}^{i_t',j_t'})$ satisfies  $\mathcal{C}_{ij}$
and $T(L+\sum_{t=1}^{s}\epsilon_t\boldsymbol{\delta}^{i_t',j_t'})=T(R+l_{ij}-r_{ij})$.

If $l_{ij}-r_{ij}=m \geq 0$, set  $(i_t,j_t)=(i_t',j_t')$ for $1\leq t \leq s$, and $(i_t,j_t)=(i,j)$ for $s+1 \leq t\leq t+m$.

If $l_{ij}-r_{ij}=m < 0$, set $(i_t,j_t)=(i,j)$ for $ 1 \leq t\leq m$ and $(i_{m+t},j_{m+t})=(i_t',j_t')$
for $1\leq t \leq s$.
\end{proof}

\begin{theorem}\label{thm-irr}
The Gelfand-Tsetlin module $V_{\mathcal{C}}(T(L))$ is simple if and only if $\mathcal{C}$ is the maximal (admissible) set of relations satisfied by $T(L)$.
\end{theorem}

\begin{proof}
Suppose
$\mathcal{C}$ is not the maximal set of relations satisfied by $T(L)$. Then
$l_{k+1, i}-l_{k, j}\in \mathbb{Z}$ for some indexes and there is no relation between $(k+1, i)$ and $(k, j)$.
So there exists a tableau $T(R)=(r_{st})\in U T(L)$ such that $r_{k+1, i}-r_{k, j}\in \mathbb{Z}_{\geq 0}$
and $T(Q)=(q_{st})\in UT(L)$ such that $q_{k , j}-q_{k+1, i}\in \mathbb{Z}_{>0}$.
By  the Gelfand-Tsetlin formulas one has that  $T(Q)$ is not in the submodule  of $V_{\mathcal{C}}(T(L))$ generated by $T(R)$ and thus $V_{\mathcal{C}}(T(L))$ is not simple.

Conversely, let $\mathcal{C}$ be the maximal set of relations satisfied by $T(L)$. By Lemma \ref{lemma sequences}, for any tableau  $ T(R)\in \mathcal{B}_{\mathcal{C}}(T(L))$, there exist $\{(i_t,j_t)\}$ $1\leq t \leq s$
such that for any $r\leq s$, $T(L+\sum_{t=1}^{r}\epsilon_t\boldsymbol{\delta}^{i_t,j_t})$ satisfies  $\mathcal{C}$
and $T(L+\sum_{t=1}^{s}\epsilon_t\boldsymbol{\delta}^{i_t,j_t})=T(R)$.
If $T(L)$ and $T(L+\boldsymbol{\delta}^{ij})$ satisfy $\mathcal{C}$,
then $l_{i,j}\neq l_{i+1,j'}$
for any $j'$. Similarly if $T(L)$ and $T(L-\boldsymbol{\delta}^{i,j})$ satisfy $\mathcal{C}$, then $l_{i,j}\neq l_{i-1,j'}$ for any $j'$. Thus the coefficient of $T(L+  \boldsymbol{\delta}^{i ,j })$ in
$E_{i,i+1 }T(L)$ (respectively $T(L-\boldsymbol{\delta}^{i ,j })$ in
$E_{i+1,i }T(L)$)
is nonzero. By Lemma \ref{lemma separation},
$T(L \pm\boldsymbol{\delta}^{i_1,j_1}) \in V_{\mathcal{C}}(T(L))$. Repeating the argument $s$ times we conclude that $T(R) \in V_{\mathcal{C}}(T(L))$.
\end{proof}

\begin{definition}
We will say that
a tableau
$T(L)$ satisfies the \it{FRZ-condition} if it is a realization of some admissible set $\mathcal{C}\subset \mathcal{R} $.

\end{definition}

Therefore, each admissible set $\mathcal{C}$ defines infinitely many  tableaux satisfying the FRZ-condition, each of which gives rise to an simple admissible Gelfand-Tsetlin module. These simple modules form infinitely many isomorphism classes.

\begin{theorem}\label{thm-mult}
For any $\sm\in \Sp \Gamma$ from the Gelfand-Tsetlin  support of $V_{\mathcal{C}}(T(L))$,  the Gelfand-Tsetlin multiplicity of $\sm$ is one.
\end{theorem}

\begin{proof}
The action of $\Gamma$ is given by the formulas (\ref{action of Gamma in finite dimensional modules}), and hence determined by the values
 of  symmetric polynomials on the entries of the rows of the tableaux. Given two Gelfand-Tsetlin tableaux $T(L)$ and $T(R)$, we have $c_{rs}(T(L))=c_{rs}(T(R))$ for any $1\leq s\leq r\leq n$ if and only if $L=\sigma(R)$ for some $\sigma\in G$. In particular, $T(R)\in\mathcal{B}_{\mathcal{C}}(T(L))$ and $L=\sigma(R)$ for some $\sigma\neq  id$ implies the existence of $T(Q)\in\mathcal{B}_{\mathcal{C}}(T(L))$ with $q_{ki}=q_{kj}$ for some $1\leq j\neq i\leq n-1$.
\end{proof}

\subsection{Proof of Theorem II}

Consider a tableau
$T(L)$ satisfying the \it{FRZ-condition} and the corresponding character $\chi_{T(L)}$ of $\Ga$. Then the tableau
$T(L)$ is a realization of some admissible set $\mathcal{C}\subset \mathcal{R} $.

 Let $\sm_L=\Ker \chi_{T(L)}$. Since $V_{\mathcal{C}}(T(L))$ has  $\sm_L$ in its Gelfand-Tsetlin support and has a basis consisting of noncritical tableaux with standard action of the generators of $\gl_n$ by Theorems \ref{sufficiency of admissible} and \ref{thm action of gamma}, it has an simple subquotient $V(L)$ satisfying Theorem II. It shows the existence. Let $W$ be the $\gl_n$-submodule of  $V_{\mathcal{C}}(T(L))$ generated by $T(L)$.
 If $V$ is an simple Gelfand-Tsetlin module having  $\sm_L$ in its Gelfand-Tsetlin support and having a basis consisting of noncritical tableaux with standard action of the generators of $\gl_n$ then $V$ is a homomorphic image of
 $W$.
 The Gelfand-Tsetlin multiplicity of $\sm_L$ in $V_{\mathcal{C}}(T(L))$ is one by Theorem \ref {thm-mult}. This shows that $V$ is the unique module with desired properties. Hence $V\simeq V(L)$  proving the uniqueness.

\subsection{Examples of admissible modules}

The following proposition gives  a family of highest weight modules that can be realized as $V_{\mathcal{C}}(T(L))$ for some admissible set of relations $\mathcal{C}$.

\begin{proposition}\label{hw module}
Set $\lambda=(\lambda_1,\ldots ,\lambda_n)$. The simple highest weight module $L(\lambda)$ is an admissible Gelfand-Tsetlin module if  $\lambda_{i}-\lambda_j\notin \mathbb{Z}$ or $\lambda_{i}-\lambda_j> i-j$ for any $1\leq i<j\leq n-1$.
\end{proposition}

\begin{proof}
Let $T(L)$ be a tableau such that $l_{ij}=\lambda_j-j+1$ and let $\mathcal{C}$ be the maximal set of relations satisfied by $T(L)$. Then $\mathcal{C}$ is admissible and $V_{\mathcal{C}}(T(L))$ is an simple highest weight module with highest weight $\lambda=(\lambda_1,\ldots ,\lambda_n)$ and $T(L)$ is a  highest weight vector.

\end{proof}


We recall the construction of generic Verma modules \cite{Maz1} which satisfy Proposition \ref{hw module}.

\begin{example}
Let $\mathcal {C}$ be the  following set of relations:
 $$\{((k+1,i);(k,i))\ |\ 1\leq i\leq k\leq n-1\}$$
and $T(L)$ be a $\mathcal {C}$-realization. By Theorem \ref{sufficiency of admissible} $\mathcal {C}$ is admissible.
Thus $V_{\mathcal{C}}(T(L))$ is a $\gl_n$-Verma module with highest weight
$\lambda=(\lambda_1,\ldots,\lambda_n)$, where $\lambda_i=l_{ni}+i-1$.

\end{example}

A weight $\gl_n$-module is called \it{dense} if for any weight $\lambda$ and any root $\alpha$ of $\gl_n$, $\lambda+ \alpha$ is also a root.
The following  dense modules  were constructed in \cite{Maz3}.

\begin{example}
Let $\mathcal {C}=\mathcal {C}^+\cup \mathcal {C}^-$ be the  following set of relations:
\begin{align*}
\mathcal {C}^+=\{((i+1,j);(i,j))\ |\ k\leq j\leq i\leq n-1\} \\
\mathcal {C}^-\{((i,j);(i+1,j+1))\ |\ k\leq j\leq i\leq n-1\},
\end{align*}
$2\leq k\leq n$, and $T(L)$ be a $\mathcal {C}$-realization. By Theorem \ref{sufficiency of admissible}, $\mathcal {C}$ is admissible.
Thus $V_{\mathcal{C}}(T(L))$ is a $\gl_n$-module.
When $k=2$,  $V_{\mathcal{C}}(T(L))$ is a dense module with finite weight multiplicities(\cite{Maz3}, Lemma 1). When $k>2$,  $V_{\mathcal{C}}(T(L))$ is a dense module with infinite weight multiplicities (\cite{Maz3}, Theorem 5).

\end{example}

\section{   Tableaux Gelfand-Tsetlin modules }
In this section we discuss the place of admissible Gelfand-Tsetlin modules among those Gelfand-Tsetlin modules which have a realization by Gelfand-Tsetlin formulas.

We start with the following natural question: if $\mathcal{C}$ is not an admissible set, is there a tableau $T(L)$   satisfying $\mathcal C$  such that $V_{\mathcal{C}}(T(L))$ is a $\gl_n$-module? The answer is positive. Here is an example.

\begin{example}
Let $\mathcal{C}$ be the set of relations  corresponding to the following  graph
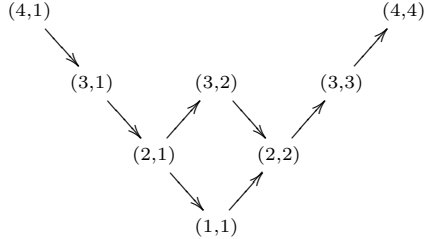
\begin{center}
\begin{tabular}{c c}
\xymatrixrowsep{0.5cm}
\xymatrixcolsep{0.1cm}
\xymatrix @C=0.1em {
 \scriptstyle{(4,1)}\ar[rd]& &   & & &   &\scriptstyle{(4,4)} \\
  & \scriptstyle{(3,1)}  \ar[rd]   & &\scriptstyle{(3,2)}\ar[rd] & &\scriptstyle{(3,3)}\ar[ru]  \\
    &     &\scriptstyle{(2,1)}\ar[ru] \ar[rd] & &\scriptstyle{(2,2)}\ar[ru]  & \\
  &    & &\scriptstyle{(1,1)} \ar[ru] & &\\
}
\end{tabular}
\end{center}
and $T(L)$ the following tableau\\
\begin{center}
\hspace{1.5cm}\Stone{$3$}\Stone{$\frac{3}{2}$}\Stone{$\frac{3}{2}$}\Stone{$0$}\\[0.2pt]
 \hspace{1.5cm}\Stone{$3$}\Stone{$2$}\Stone{$1$}\\[0.2pt]
 \hspace{1.1cm} \Stone{$3$}\Stone{$2$}\\[0.2pt]
 \hspace{1.4cm}\Stone{$3$}\\
\end{center}

\

Then $T(L)$ satisfies $\mathcal{C}$ and $V_{\mathcal{C}}(T(L))$ is a $1$-dimensional $\gl_4$-module but the action of $\Gamma$ is not given by (\ref{action of Gamma in finite dimensional modules}).
\end{example}

This example suggests the following definition.

\begin{definition}
 We say that a Gelfand-Tsetlin $\mathfrak{gl}_{n}$-module $V$ is a \it{tableaux module} if it satisfies the following conditions:
\begin{itemize}
\item[(i)] $V$  has a basis  consisting of noncritical tableaux.
\item[(ii)] The action of $\mathfrak{gl}_{n}$ on $V$ is given by the Gelfand-Tsetlin formulas (\ref{Gelfand-Tsetlin formulas}).
\item[(iii)] All Gelfand-Tsetlin multiplicities of $V$ are  bounded by $1$.
\item[(iv)] The action  of $\Gamma$ on $V$ is given by the formula \eqref{action of Gamma in finite dimensional modules}.
  \end{itemize}
  \end{definition}

Hence, the module $V_{\mathcal{C}}(T(L))$ from the example above is not a tableaux module.
As we showed in the previous section any admissible Gelfand-Tsetlin module is a tableaux module. We believe that the converse  also holds, namely
we state the following conjecture.

\begin{conjecture}\label{conj-tableau}
If $V$ is a tableaux Gelfand-Tsetlin  $\gl_{n}$-module
then $V$ is isomorphic to some $V_{\mathcal{C}}(T(L))$
with admissible $\mathcal {C}$.
\end{conjecture}

The conjecture is known to be true for $n=3$ (\cite{FGR1}, Remark 9.2). In the rest of this section we justify the conjecture for $n=4$.

\begin{lemma}\label{Lemma relation set}
Let $\mathcal{C}$ be an admissible set, $T(L)$  a tableau satisfying $\mathcal{C}$. If there is no tableau $T(R\pm\boldsymbol{\delta}^{n-1,r})$ satisfying $\mathcal{C}$ for any $1\leq r \leq n-1$ where $r_{ij}=l_{ij}$ for $1\leq j\leq i$, $n-1\leq i\leq n$, then $\mathcal{C}=\mathcal{S}$ and $l_{ni}-l_{n,i+1}=1$ for
$1\leq i\leq n-1$.
\end{lemma}
\begin{proof}
Suppose the entries in the $(n-1)$-th row are contained in $m$ disconnected subsets of $\mathcal{C}$.
Let $\mathcal{C}_1$ be a subset of $\mathcal{C}$. Assume $(n-1,j),1\leq j \leq j_1$ are contained in
$\mathfrak{V}(\mathcal{C}_1)$. By Condition \eqref{condition for admissible}, there are at least $j_1-1$ entries  of the $n$-th row  in
$\frak{V}(\mathcal{C}_1)$. Moreover, if  $((n ,s);(n-1,1))$ (respectively $ ((n-1, j_1); (n, s))$) is not in $\mathcal C_1$ then
$T(R\pm\boldsymbol{\delta}^{n-1,1})$ (respectively $T(R\pm\boldsymbol{\delta}^{n-1, j_1}))$ satisfies $\mathcal{C}$ which is a contradiction. Hence,  $\frak{V}(\mathcal C_1)$ contains at least $j_1+1$ elements of the $n$-th row. Thus,  all the entries of the $(n-1)$-th row are contained in the same disconnected subset of $\mathcal{C}$ and the relations between the $(n-1)$-th row and the $n$-th row are as follows:
$$\{((n,i);(n-1,i)),\ ((n-1,i);(n,i+1))\ |\ 1\leq i\leq n-1 \}.$$
Therefore $\mathcal{C}=\mathcal{S}$ by  Condition \eqref{condition for admissible}.
\end{proof}

\begin{proposition}\label{nonadmissible}
Let $n=4$, $\mathcal{C}$ be a non admissible set and $T(L)$ be a tableau satisfying $\mathcal{C}$.
If $\mathcal{C}$ is the maximal set of relations satisfied by  $T(L)$
and the Gelfand-Tsetlin formulas \eqref{Gelfand-Tsetlin formulas} define a $\gl_n$-module structure on $V_{\mathcal{C}}(T(L))$, then
$V_{\mathcal{C}}(T(L))$ is not  a tableaux module, that is the action of $\Gamma$ is not given by \eqref{action of Gamma in finite dimensional modules}.
\end{proposition}

\begin{proof}
Suppose $k$ is the minimal such that there exist $(k,i)$, $(k,j)$ that do not satisfy the Condition \eqref{condition for admissible}.
Then for any fixed top row there exists a tableau $T(L)$ such that $l_{ki}-l_{kj}=1$ and
$\#\{i'\ |\ l_{k+1,i'}=l_{kj}\}+\#\{j' |\  l_{k-1,j'}=l_{ki}\}=1$.

Now we consider $\gl_{k+1}$-module generated by $T(L)$. We have

\begin{equation}\label{coeff}
\begin{split}
    [e_k,f_k]T(R)
    =&\sum_{r,s=1}^{k}\Phi(L,-\boldsymbol{\delta}^{kr})f_{kr}(L)e_{ks}(L+\boldsymbol{\delta}^{kr})T(L-\boldsymbol{\delta}^{kr}+\boldsymbol{\delta}^{ks})\\
-&\sum_{r,s=1}^{k} \Phi(L,\boldsymbol{\delta}^{ks})e_{ks}(L)f_{kr}(L+\boldsymbol{\delta}^{ks})T(L-\boldsymbol{\delta}^{kr}+\boldsymbol{\delta}^{ks}).\\
\end{split}
\end{equation}
     The same formula holds for the tableau $T(v)$ with distinct variable entries (considered as a generic tableau).
     Then the coefficient of $T(L)$ in \eqref{coeff} is the same as the coefficient of $T(L)$ in $h_kT(L)$ if and only if
     \begin{equation}\label{coefficient}
      \begin{split}
\lim_{v\rightarrow l}\left( \sum_{\Phi(L,-\boldsymbol{\delta}^{jr})=0} f_{kr}(v)e_{kr}(v+\boldsymbol{\delta}^{ir})
- \sum_{\Phi(L,\boldsymbol{\delta}^{ir})=0}  e_{kr}(v)f_{kr}(v+\boldsymbol{\delta}^{kr})\right)=0.\\
\end{split}
\end{equation}

     By Lemma \ref{action of center1}, (i) and (ii) one has
  \begin{align*}
\lim_{v\rightarrow l}\left( \sum_{\Phi(L,-\boldsymbol{\delta}^{kr})=0} f_{kr}(v)e_{kr}(v+\boldsymbol{\delta}^{kr})
- \sum_{\Phi(L,\boldsymbol{\delta}^{kr})=0}  e_{kr}(v)f_{kr}(v+\boldsymbol{\delta}^{kr})\right)\\
  =\lim_{v\rightarrow l}
  \left( \sum f_{k,i}(v)e_{k,i}(v+\boldsymbol{\delta}^{k,i})
 - e_{k,j}(v)f_{k,j}(v+\boldsymbol{\delta}^{k,j})\right),\\
\end{align*}
 where  the sum runs  over all  pairs $(i,j)$ such that $l_{ki}-l_{kj}=1$ and $(k,i),(k,j)$ do not satisfy condition
\eqref{condition for admissible}. If there is only one such  pair $(k,i), (k,j)$ then by direct computation we obtain
$\lim\limits_{v\rightarrow l} f_{k,i}(v)e_{k,i}(v+\boldsymbol{\delta}^{k,i})
 - e_{k,j}(v)f_{k,j}(v+\boldsymbol{\delta}^{k,j})\neq 0$.
 If there are two such pairs then there exists  $T(Q)=T(R\pm\boldsymbol{\delta}^{k-1,s})$ satisfying $\mathcal{C}$ where $r_{ij}=l_{ij}$ for $1\leq j\leq i$, $k-1\leq i\leq k+1$.
For every such pair $(k,i), (k,j)$,  $$\lim\limits_{v\rightarrow l} f_{k,i}(v)e_{k,i}(v+\boldsymbol{\delta}^{k,i})
 - e_{k,j}(v)f_{k,j}(v+\boldsymbol{\delta}^{k,j})$$  can be written as
 $(l_{k-1,s}-a)b$, where $a=l_{kj}$, $b$ is a rational function in $l$.
 Since $V_{\mathcal{C}}(T(L))$ is a module, one has
$ (l_{k-1,s}-a_1)b_1+(l_{k-1,s} -a_2)b_2=0$ and
$ (l_{k-1,s}\pm1-a_1)b_1+(l_{k-1,s}\pm1 -a_2)b_2=0$.
Then $b_1+b_2=0$, $a_1b_1+a_2b_2=0$ and $a_1\neq a_2$, Thus $b_1=b_2=0$ which is a contradiction.
Hence, there is no tableau $T(Q)=T(R\pm\boldsymbol{\delta}^{k-1,s})$ satisfying $\mathcal{C}$ where $r_{ij}=l_{ij}$ for $1\leq j\leq i$, $k-1\leq i\leq k+1$.
By Lemma \ref{Lemma relation set}, $U(\gl_{k+1})T(L)$ is one dimensional with unique tableau $T(L)$.
Let $T(L')$ be a tableau with $l_{st}'=l_{st}$ for $1\leq t\leq s \leq k$, $l_{k+1,t}'=l_{k,t}$ for $1\leq t \leq k$ and
$l_{k+1,k+1}'=l_{k,k}+1$. One has  $c_{k+1,t}T(L)=\gamma_{k+1,t}(L')T(L)$, $t=1, \ldots, k+1$. We see that the action of $\Gamma$ is different from
 \eqref{action of Gamma in finite dimensional modules}. Thus the $1$-dimensional module $V_{\mathcal{C}}(T(L))$ is not
  a tableaux     module.
\end{proof}

\begin{remark} Let $\mathcal{C}$ be any set of relations.
\begin{itemize}
\item[(i)]
 If $\mathcal{C}$ a non admissible set of relations  then $V_{\mathcal{C}}(T(L))$ is not a $\gl_n$-module for some $T(L)$ satisfying $\mathcal{C}$. In the case  $n=3$, $V_{\mathcal{C}}(T(L))$ is not a $\gl_n$-module for any $T(L)$ for which $\mathcal{C}$ is the maximal set of relations.
\item[(ii)]
 $V= span_{\C} \{T(L)=(3,0|2)\}$ is a $1$-dimensional tableaux module. Its basis is not equal to any set
${\mathcal B}_{\mathcal{ C}}(T(L))$.
\end{itemize}
\end{remark}




\section{Gelfand-Graev continuation}

In  this section we prove  necessary and sufficient condition for the Gelfand and Graev's continuation.

For the sake of convenience we will use our notation to describe Gelfand and Graev's continuations.
In  \cite{GG} the standard
labelling of tableaux given in \cite{GT} is slightly modified and the action of generating elements of the Lie algebra is given
on this new basis.  To each $k=1, 2,\ldots, n-1$ we assign a pair of integers $\{i_k,i_k'\}$ where $i_k\in\{0,1,\ldots,k\}, i'_k\in\{1,2,\ldots,k+1\}$, and $i_k<i_k'$.
For each such set of indexes it is constructed a Hilbert space
$H\{i_k,i_k'\}$ having an orthonormal basis labeled by the set of all possible tableaux $T(L)$ with integral entries where the top row is fixed and the other components satisfy the following set of inequalities:
{
\begin{align}
&l_{k,j}>l_{k,j+1},                    & j<k\leq n\\
&l_{k+1,j-1}\geq l_{kj}> l_{k+1,j},    & j\leq i_k\\
&l_{k+1,j}\geq l_{kj}> l_{k+1,j+1},    & i_k<j< i_k'\\
&l_{k+1,j+1}\geq l_{kj}> l_{k+1,j+2},  & j\geq i_k'.
\end{align}
}
    Let $\mathcal {C}$ be the  following set of relations:
\begin{align*}
&\{((n,i);(n,i+1))|i=i_{n-1}',i_{n-1}'\}\\
\cup&\{(k+1,j-1);(k,j)), ((k,j);(k+1,j))\ |\ j\leq i_{k}\}\\
\cup&\{((k+1,j);(k,j)),((k,j),(k+1,j+1))\ |\ i_{k}< j< i'_{k}\}\\
\cup& \{((k+1,j+1);(k,j)),((k,j);(k+1,j+2))\ |\ i'_{k}\leq j\}.
\end{align*}

\begin{remark} Every tableau $T(L)$ satisfying the GG-condition is a $\mathcal{C}$-realization. Conversely, not every
$\mathcal{C}$-realization satisfies the  GG-condition.
For example, if $n=4$, $i_t=0, i_t'=t+1$ for $t=1,2$, $i_3=1$, $i_3'=3$ then
the graph associated with $\mathcal{C}$  is as follows
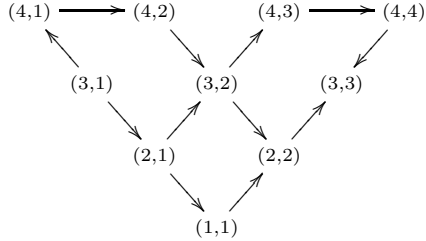
\begin{center}
\begin{tabular}{c c}
\xymatrixrowsep{0.5cm}
\xymatrixcolsep{0.1cm}\xymatrix @C=0.1em {
 \scriptstyle{(4,1)}\ar[rr]& &\scriptstyle{(4,2)} \ar[rd]   & &\scriptstyle{(4,3)}\ar[rr] &   &\scriptstyle{(4,4)}\ar[ld] \\
  & \scriptstyle{(3,1)} \ar[lu] \ar[rd]   & &\scriptstyle{(3,2)}\ar[rd]\ar[ru]  & &\scriptstyle{(3,3)} \\
    &     &\scriptstyle{(2,1)}\ar[ru] \ar[rd] & &\scriptstyle{(2,2)}\ar[ru]  & \\
  &    & &\scriptstyle{(1,1)} \ar[ru] & &\\
}\\

\end{tabular}
\end{center}
A tableau $T(L)$ with $l_{41}=l_{42}$ and $l_{43}=l_{44}$ is  a $\mathcal{C}$-realization but
 it does not satisfy the GG-condition.
\end{remark}

Lemire and Patera \cite{LP} gave counterexamples and showed that for certain sets of indexes
the Gelfand-Tsetlin formulas  do not define $\gl_n$-module. In fact, they claimed (though without proof) that a necessary
condition to have a $\gl_n$-module on
$H\{i_k,i_k'\}$  is that for each $k=2, 3,\ldots, n-1$  one has
\begin{equation}\label{LP condition}
i_{k-1},i_{k-1}'\in\{0, i_k,i_k'-1, k\}.
\end{equation}

We will call  the {\bf LP-condition}, the GG-condition together with this restriction. In \cite{LP} it was given an example of a tableau for $\gl_3$ which does not satisfy the
LP-condition and does not generate a module. This is not sufficient  to conclude that $H\{i_k,i_k'\}$ is not  a module if it contains a basis tableau  which does not
satisfies the LP-condition, since  $H\{i_k,i_k'\}$ not only depends on the choice of pairs $i_k$, $i_k'$
but also on the top row of the tableau.  Lemire and Patera showed that the LP-condition is sufficient to have a module structure on
$H\{i_k,i_k'\}$  for $\gl_3$ and in some cases for $\gl_4$.

The following is clear

\begin{lemma}\label{choice of tab} Fix $i_k,i_k'$.
If $r_{ij}$ $i\geq k$, $1\leq j\leq i$ satisfy the GG-condition  then there exists
a tableau  in  $H\{i_k,i_k'\}$ that satisfies the GG-condition and $l_{ij}=r_{ij}$ $i\geq k$ $1\leq j\leq i$.
\end{lemma}

\

\noindent{\bf Proof of Theorem I.}
If  $i_k, i_k'$ satisfy {\bf LP-condition}, then $\mathcal{C}$ is admissible by Theorem \ref{sufficiency of admissible}. Any tableau satisfying GG-condition is a $\mathcal{C}$-realization.
Thus $H\{i_k,i_k'\}\simeq V_{\mathcal{C}}(T(L))$ is a $\gl_n$-module.

Now we show that the LP-condition is necessary.
Let  $k$ be the maximal such that $\{i_{k-1},i_{k-1}'\}\varsubsetneq\{0, i_k,i_k'-1, k\}$.
\begin{itemize}
\item[(i)] Assume that one of $\{i_{k-1},i_{k-1}'\}$ is not in $\{0, i_k,i_k'-1, k\}$. Without loss of generality we assume that
$i_{k-1}\notin \{0, i_k,i_k'-1, k\}$. By Lemma \ref{choice of tab} for any fixed top row there exists a tableau that satisfies the GG-condition in
$H\{i_k,i_k'\}$ and $l_{k,i_{k-1}-1}-l_{k,i_{k-1}}=1$,
    $l_{k-1,j}\neq l_{k,i_k}$ for any $j$. The pair
    $(l_{k,i_{k-1}-1}, l_{k,i_{k-1}})$ is the only pair in the $k$-th row such that
    $l_{k,i_{k-1}-1}-l_{k,i_{k-1}}=1$ and

   $$\#\{i'\ |\ l_{k+1,i'}=l_{k,i_{k-1}-1},\}+
   \#\{j'\ |\ l_{k-1,j'}=l_{k,i_{k-1}}\}=1.$$
    By direct computation one has

         \begin{equation}
      \begin{split}
\lim_{v\rightarrow l}\left( \sum_{\Phi(L,-\boldsymbol{\delta}^{jr})=0} f_{kr}(v)e_{kr}(v+\boldsymbol{\delta}^{ir})
- \sum_{\Phi(L,\boldsymbol{\delta}^{ir})=0}  e_{kr}(v)f_{kr}(v+\boldsymbol{\delta}^{kr})\right)\\
=\lim_{v\rightarrow l}
  \left(   f_{k,i_{k-1}-1}(v)e_{k,i_{k-1}-1}(v+\boldsymbol{\delta}^{i,i_{k-1}-1})
 - e_{k,i_{k-1}}(v)f_{k,i_{k-1}}(v+\boldsymbol{\delta}^{k,i_{k-1}})\right)\\
\end{split}
\end{equation}
     which is nonzero. Thus $H\{i_k,i_k'\}$ is not a module.

\item[(ii)] Suppose both of $\{i_{k-1},i_{k-1}'\}$ are not in $\{0, i_k,i_k'-1, k\}$.
By Lemma \ref{choice of tab} for any fixed top row there exists a tableau that satisfies the GG-condition in  $H\{i_k,i_k'\}$
 and $l_{k,i_{k-1}-1}-l_{k,i_{k-1}}=1$, $l_{k,i_{k-1}'+1}-l_{k,i_{k-1}'+2}=1$ and
    $l_{k-1,j}\neq l_{k,i_{k-1}}$, $l_{k-1,j}\neq l_{k,i_{k-1}'+2}$ for any $j$.
  By Lemma \ref{action of center1}, (i) and (ii), one has
\begin{equation}
      \begin{split}
\lim_{v\rightarrow l}\left( \sum_{\Phi(L,-\boldsymbol{\delta}^{jr})=0} f_{kr}(v)e_{kr}(v+\boldsymbol{\delta}^{ir})
- \sum_{\Phi(L,\boldsymbol{\delta}^{ir})=0}  e_{kr}(v)f_{kr}(v+\boldsymbol{\delta}^{kr})\right)\\
 =\lim_{v\rightarrow l}\left(\sum_{r=i_{k-1}+1,i_{k-1}'+2}f_{kr}(v)e_{kr}(v+\boldsymbol{\delta}^{ir})
-\sum_{r=i_{k-1},i_{k-1}'+1}  e_{kr}(v)f_{kr}(v+\boldsymbol{\delta}^{kr})\right).\\
\end{split}
\end{equation}
By Lemma \ref{choice of tab}, $l_{k-1,i_{k}}$ has at least $2$ choices if we fix all other $l_{st}$, $1\leq t\leq s$, $k-1\leq s \leq n$. Applying same argument as in the proof of Proposition \ref{nonadmissible} one can show that it is impossible to have zero limit  for all these tableaux. Thus $H\{i_k,i_k'\}$ is not a module.
\end{itemize}

%




\end{document}